\documentclass[onefignum,onetabnum]{siamonline171218}

\usepackage{mathtools}
\usepackage[]{overpic}
\usepackage{bm}
\usepackage{subcaption}
\usepackage[usestackEOL]{stackengine}
\usepackage{booktabs}
\usepackage{amsmath,amsfonts,amssymb}
\usepackage{listofsymbols}
\usepackage{caption}
\usepackage{multirow}
\usepackage{makecell}

\usepackage{hyperref}
\usepackage{framed}
\usepackage{nomencl}
\makenomenclature

\opensymdef
\newsym[Energy]{symE}{E}
\newsym[Mass]{symm}{m}
\newsym[Speed of light]{symc}{c}
\closesymdef

\DeclareMathOperator*{\argmin}{argmin}
\newcommand{\norm}[1]{\left\lVert#1\right\rVert}
\DeclarePairedDelimiter\abs{\lvert}{\rvert}

\newcounter{casenum}
\newenvironment{caseof}{\setcounter{casenum}{1}}{\vskip.5\baselineskip}
\newcommand{\case}[2]{\vskip.5\baselineskip\par\noindent {\bfseries Case \arabic{casenum}:} #1\\#2\addtocounter{casenum}{1}}

\usepackage{pifont}
\newcommand{\cmark}{\ding{51}}%
\usepackage{lipsum}
\usepackage{amsfonts}
\usepackage{graphicx}
\usepackage{epstopdf}
\usepackage{algorithmic}
\ifpdf
  \DeclareGraphicsExtensions{.eps,.pdf,.png,.jpg}
\else
  \DeclareGraphicsExtensions{.eps}
\fi

\usepackage{enumitem}
\setlist[enumerate]{leftmargin=.5in}
\setlist[itemize]{leftmargin=.5in}

\newsiamremark{remark}{Remark}
\newsiamremark{hypothesis}{Hypothesis}
\crefname{hypothesis}{Hypothesis}{Hypotheses}
\newsiamthm{claim}{Claim}

\headers{Centering Data Improves the Dynamic Mode Decomposition}{S. M. Hirsh, K. Decker Harris, J. N. Kutz and B. W. Brunton}

\title{Centering Data Improves the Dynamic Mode Decomposition}

\author{Seth M. Hirsh\thanks{Department of Physics, University of Washington, Seattle, WA 
  (\email{hirshs@uw.edu}).}
\and Kameron Decker Harris\footnotemark[4]~\thanks{Paul G. Allen School of Computer Science, University of Washington, Seattle, WA 
  (\email{kamdh@uw.edu}).}
\and J. Nathan Kutz\thanks{Department of Applied Mathematics, University of Washington, Seattle, WA 
(\email{kutz@uw.edu})}
\and Bingni W. Brunton\thanks{Department of Biology, University of Washington, Seattle, WA 
  (\email{kamdh@uw.edu},\email{bbrunton@uw.edu}).}
}

\usepackage{amsopn}

\ifpdf
\hypersetup{
  pdftitle={Centering Data Improves the Dynamic Mode Decomposition }
  pdfauthor={S. M. Hirsh, K. D. Harris, J. N. Kutz and B. W. Brunton}
}
\fi

\DeclarePairedDelimiterX\set[1]\lbrace\rbrace{#1}

\newcommand{\R}{\mathbb{R}}

\newcommand{\A}{\bar{\bm{A}}}

\begin{document}
\maketitle

\begin{abstract}
Dynamic mode decomposition (DMD) is a data-driven method that models high-dimensional time series as a sum of spatiotemporal modes, where the temporal modes are constrained by linear dynamics.
For nonlinear dynamical systems exhibiting strongly coherent structures, DMD can be a useful approximation to extract dominant, interpretable modes.
In many domains with large spatiotemporal data---including fluid dynamics, video processing, and finance---the dynamics of interest are often perturbations about fixed points or equilibria, which motivates the application of DMD to centered (i.e. mean-subtracted) data.
In this work, we show that DMD with centered data is equivalent to incorporating an affine term in the dynamic model and is not equivalent to computing a discrete Fourier transform.
Importantly, DMD with centering can always be used to compute eigenvalue spectra of the dynamics.
However, in many cases DMD without centering cannot model the corresponding dynamics, most notably if the dynamics have full effective rank.
Additionally, we generalize the notion of centering to extracting arbitrary, but known, fixed frequencies from the data. 
We corroborate these theoretical results numerically on three nonlinear examples: 
the Lorenz system, a surveillance video, and brain recordings. 
Since centering the data is simple and computationally efficient, we recommend it as a preprocessing step before DMD; furthermore, we suggest that it can be readily used in conjunction with many other popular implementations of the DMD algorithm.

\end{abstract}




\section{Introduction}

Recent advances in sensing, data storage, and computing technologies have resulted in an unprecedented increase in the availability of large-scale measurements. 
Many measurements come from high-dimensional, complex systems in which the governing equations are poorly understood or entirely unknown, which has motivated the development of data-driven techniques for characterizing and modeling  spatiotemporal dynamics.
Importantly, these techniques must be computationally efficient and interpretable, providing insights into the underlying physics and potentially enabling predictions for rapid manipulation and control.

One popular method for modeling such systems is the {\em dynamic mode decomposition} (DMD) \cite{schmid2008sixty,schmid2010dynamic,rowley2009spectral,mezic2013analysis,tu2013dynamic,kutz2016dynamic}. 
Like principal component analysis (PCA)~\cite{jolliffe2011principal,wold1987principal} and independent component analysis (ICA)~\cite{hyvarinen1999fast}, DMD is a dimensionality reduction technique that decomposes data into a set of spatial and temporal modes. Unlike PCA and ICA, DMD makes the additional assumption that the data are observations from an underlying dynamical system.
In particular, the dynamics are assumed to be approximately linear, and the data are decomposed into pairs of interpretable spatial and temporal modes. 
DMD has been successfully applied in a wide variety of disciplines, including fluid dynamics~\cite{schmid2011applications}, neuroscience~\cite{brunton2016extracting}, 
disease modeling~\cite{proctor2015discovering}, finance~\cite{mann2016dynamic}, and computer vision~\cite{grosek2014dynamic}. 
In addition, several extensions and variations to the DMD algorithm have been developed (see \cite{jovanovic2012low,williams2015data,kutz2016multiresolution,zhang2017online,proctor2016dynamic,jovanovic2014sparsity,askham2018variable}, among many others).

For many systems of interest, the dynamics we want to model are perturbations about equilibria.
To name a few specific examples, in hydrodynamics we may model motion of a fluid about a base flow~\cite{noack2003hierarchy,tadmor2010mean}; in video processing we may extract the foreground from a static background~\cite{sobral2014comprehensive}; and in climate science we may analyze anomalies that depart from long-term averages~\cite{hart2001using,fritts1971multivariate}. 
Further, linearizing about equilibria provides key information on the stability of the system about these fixed points.
In general, the mean of the measurement data is a natural estimate of an unknown equilibrium point;
therefore, it is natural to apply DMD on mean-subtracted data.


\begin{table}[]
\centering
\begin{tabular}{llcc}
 &  & \multicolumn{2}{c}{Data Generation} \\
 &  & \thead{Linear System \\ $\bm{x}_{j+1} = \bm{A} \bm{x}_j$} & \thead{Affine System \\ $\bm{x}_{j+1} = \bm{A} \bm{x}_j + \bm{b}$} \\ \cline{3-4} 
\multicolumn{1}{c}{\multirow{2}{*}{\rotatebox{90}{Method}}} & \multicolumn{1}{l|}{\rotatebox{90}{\thead{DMD w/o \\ centering}}} & \multicolumn{1}{c|}{\begin{tabular}[c]{@{}c@{}}\cmark \\ (Theorem \ref{theorem:uniqueness}) \vspace{0.5cm} \end{tabular}} & \multicolumn{1}{c|}{\begin{tabular}[c]{@{}c@{}}Sometimes \\ (Theorem \ref{theorem:dmd_conditions_affine}) \vspace{0.5cm} \end{tabular}} \\ \cline{3-4} 
\multicolumn{1}{c}{} & \multicolumn{1}{l|}{\rotatebox{90}{\thead{DMD w/ \\ centering}}} & \multicolumn{1}{c|}{\begin{tabular}[c]{@{}c@{}}\cmark \\ (Theorem \ref{theorem:one_rank_update_1}) \vspace{0.5cm} \end{tabular}} & \multicolumn{1}{c|}{\begin{tabular}[c]{@{}c@{}}\cmark \\ (Theorem \ref{theorem:rank_one_update}) \vspace{0.5cm} \end{tabular}} \\ \cline{3-4} 
\end{tabular}
\caption{Comparison of performance of DMD with and without centering. 
A \cmark ~indicates that the method does correctly extract the spectrum and modes of the system in each column. 
}
\label{table:comparing_dmd_with_and_without_centering}
\end{table}

In a complementary perspective, we may think of DMD computed over a short time window as a multivariate Taylor expansion of the dynamics.
It follows that the model should include an affine, or bias, term (Fig.~\ref{fig:motivation}), which is usually not a part of the DMD model; if DMD is computed on centered data, then this affine term is expected to be small
(in fact, one of our results is that it will be zero).

In this work, we show that centering data improves the performance of DMD. 
Previous work has suggested that computing the DMD of centered data may be restrictive and have undesirable consequences~\cite{chen2012variants}.
In particular, Chen \emph{et al.}~\cite{chen2012variants} show that DMD on mean-subtracted data is equivalent to a temporal discrete Fourier transform (DFT) in time, restricting the frequencies extracted to be independent of the dataset. 
This argument hinged on the mean-subtracted data being full rank; however, here we show that, in linear systems that contain a fixed point, mean-subtracted data will always have linearly dependent columns.
Therefore, DMD on centered data does not converge to the DFT.
Furthermore, 
our proposed method of centering the data successfully extracts the equilibrium and dynamics about this equilibrium.


In Section~\ref{sec:background} we review the DMD algorithm, focusing on comparing the SVD-based approach to the companion matrix approach. 
We propose centering the data in Section~\ref{sec:centering_data},
showing that it is equivalent to incorporating an affine term in the DMD model.
Section~\ref{sec:Uniqueness} concerns the uniqueness of the DMD modes and whether the DMD problem is well-posed, 
generalizing previous results to the case where data may be low rank.
Section~\ref{sec:comparing_without_centering} compares DMD with and without centering, including theory and numerical examples. 
We find that, in the case of linear dynamics about an equilibrium point, DMD with centering can always extract the correct dynamics; however, DMD without centering sometimes produces an inaccurate model. These results are summarized in Table \ref{table:comparing_dmd_with_and_without_centering}.
The work by Chen \emph{et al.}~\cite{chen2012variants} is discussed in detail in Section~\ref{sec:DFT}, where we argue that DMD with centering is not equivalent to a DFT.{}
This notion of data centering is generalized in Section~\ref{sec:arbitraryfrequencies} to extract dynamics 
while subtracting any known fixed frequencies.
Finally, Section~\ref{sec:examples} demonstrates DMD with centering and fixed frequency subtraction on three nonlinear examples, the Lorenz system, background-foreground separation of a video, and brain recordings. As a practical recommendation, we suggest centering data as a preprocessing step in DMD.
All the code used to reproduce results in the figures is openly available at \url{https://github.com/sethhirsh/DMD_with_Centering}.


\begin{figure}[bt]
\centering
\vspace{5pt}
\begin{overpic}[scale=0.6,unit=1bp]{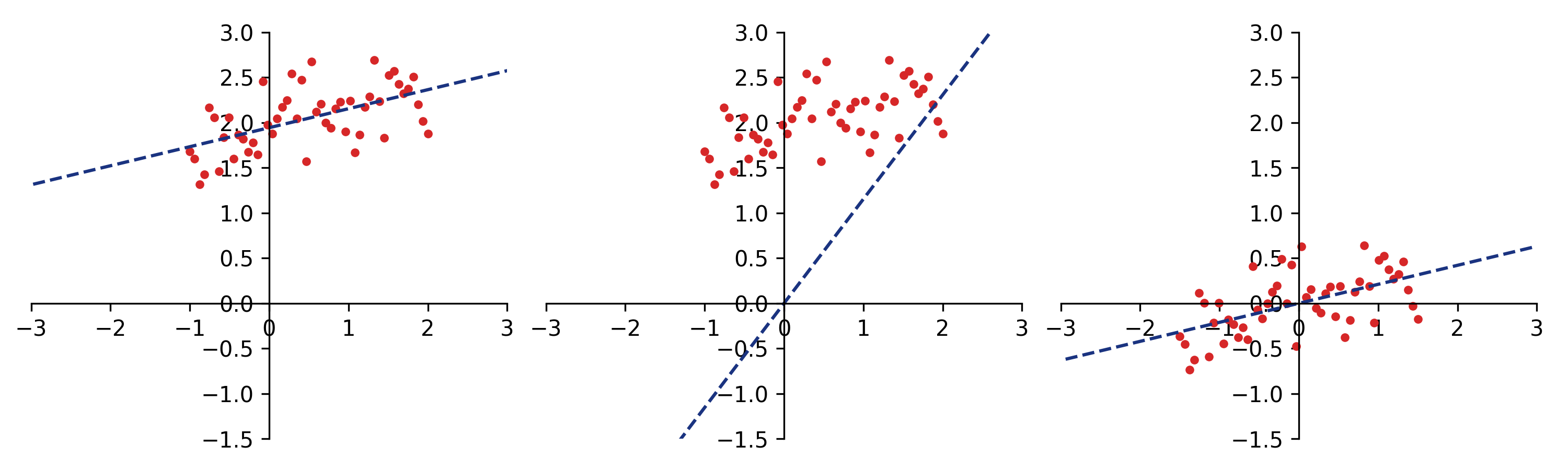}
\put(32,7){$x$}
\put(18,27.5){$y$}
\put(65,7){$x$}
\put(51,27.5){$y$}
\put(98,7){$\bar{x}$}
\put(84,27.5){$\bar{y}$}
\put(18,17){\textcolor[RGB]{31, 119, 180}{$y=ax+b$}}
\put(1,30){a)}
\put(33,30){b)}
\put(67,30){c)}
\put(54,12){\textcolor[RGB]{0, 107, 164}{$y=ax$}}
\put(86,17){\textcolor[RGB]{0, 107, 164}{$\bar{y}=a \bar{x}$}}
\end{overpic}
\caption{An illustration of the benefit of centering for one-dimensional regression, where the data $(x_j, y_j)$ is generated by an affine model with noise. a) Data fit to affine model $y = a x + b$ yields a good fit. b) Data fit to linear model $y = a x$ yields a poor fit. c) Centered data $(\bar{x}_j \bar{y}_j)$ fit to linear model $\bar{y} = a \bar{x}$ yields a good fit.}
\label{fig:motivation}
\end{figure}
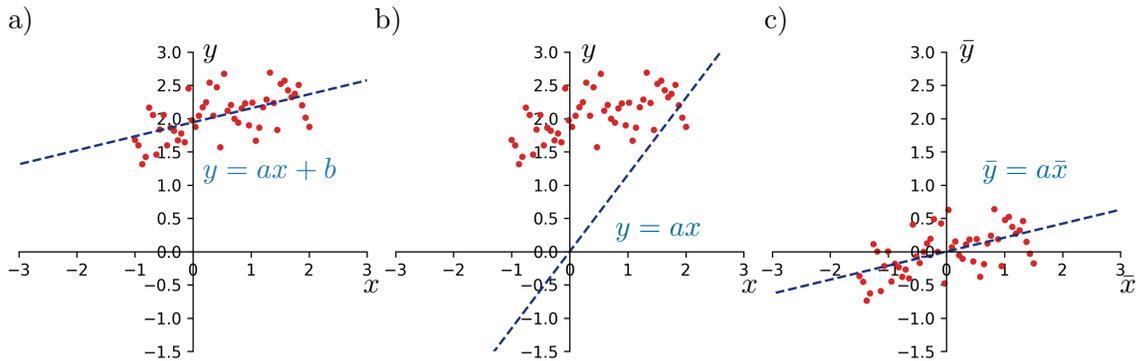

\nomenclature{Fig.}{Figure}
\nomenclature{$A_i$}{Area of the $i^{th}$ component}\textbf{}

\section{Background}
\label{sec:background}

Initially developed in the fluid dynamics community, dynamic mode decomposition (DMD) has become a popular tool for analyzing large-scale dynamical systems in many different application domains~\cite{kutz2016dynamic,schmid2011applications}. 
In this section we briefly review two formulations of this problem.

Consider a set of $T+1$ measurement snapshots $\bm{x}_j \in \R^n$ for $j = 1, \ldots, T+1$,
which are generated by linear dynamics,
\begin{equation}
  \bm{x}_{j + 1} = \bm{A} \bm{x}_{j}.
  \label{eq:approx}
\end{equation}
The goal of DMD is to characterize the dynamics of the system by the eigendecomposition of the linear operator $\bm{A} \in \R^{n \times n}$: 
\begin{equation}
  \bm{A} \bm{v}_i = \lambda_i \bm{v}_i \text{ for } i = 1, \ldots,n.
  \label{eq:DMD_modes}
\end{equation}
The eigenvectors $\bm{v}_i$ are typically refered to as the DMD modes. 
For our theoretical results, we typically assume that 
the eigenvalues $\lambda_i \neq 0$ are distinct.
For many systems of interest, the true dynamics may be nonlinear and/or stochastic.
In addition, observations may contain measurement noise.
Where the measurements deviate from true linear dynamics, the goal of DMD is to find the best linear approximation.


\begin{table}[t]
\captionsetup{labelformat=empty} 
\begin{framed}
\centering
{\HLtext \textcolor{header1}{\normalsize{Notation}}} \\
\vspace{2 pt}
\begin{tabular}{cp{0.6\textwidth}}  
$T+1$ & number of time samples \\
$n$ & number of features \\
$\bm{A}$ & $\mathbb{R}^{n \times n}$ matrix that generates a dynamical system \\
$r$ & rank of $\bm{A}$ \\
$\bm{X}$ & $\R^{n \times T+1}$ set of measurement snapshots ranging in time from $\bm{x}_1$ through $\bm{x}_{T+1}$ \\
$\bm{X}_1$ & $\R^{n \times T}$ matrix containing $\bm{x}_1$ through $\bm{x}_T$ \\
$\bm{X}_2$ & $\R^{n \times T}$ matrix containing $\bm{x}_2$ through $\bm{x}_{T+1}$ \\
$\bm{\mu}$ & $\R^n$ mean of $\bm{X}$ \\
$\bm{\mu}_1$ & $\R^n$ mean of $\bm{X}_1$ \\
$\bm{\mu}_2$ & $\R^n$ mean of $\bm{X}_2$ \\
$\hat{\bm{A}}$ & $\R^{n \times n}$ matrix computed using SVD-based DMD without centering \\
$\bar{\bm{A}}$ & $\R^{n \times m}$ matrix computed using SVD-based DMD on centered matrices $\bm{X}_1 - \bm{\mu}_1 \bm{1}^{\intercal}$ and $\bm{X}_2 - \bm{\mu}_2 \bm{1}^{\intercal}$ \\
$\bm{C}$ & $\R^{T \times T}$ companion matrix \\
$\bm{b}$ & $\R^n$ affine or bias term in dyanmics \\
$\bm{1}$ & $\R^T$ vector of ones \\
$\bm{I}$ & identity matrix
\end{tabular}
\caption{}
\vspace{-10mm}
\end{framed}
\end{table}

\subsection{SVD-based DMD}
We first summarize the most commonly used formulation of DMD, the SVD-based approached also known as exact DMD~\cite{tu2013dynamic}.
First, let us define the pair of 
snapshot matrices containing the measurement vectors
\begin{equation}
\bm{X}_1 = \begin{bmatrix} | & | &  & | \\
          \bm{x}_1 & \bm{x}_2 & \cdots & \bm{x}_{T} \\
          | & | & & |
\end{bmatrix}
\quad
\text{and}
\quad
\bm{X}_2 = \begin{bmatrix} | & | &  & | \\
          \bm{x}_2 & \bm{x}_3 & \cdots & \bm{x}_{T + 1} \\
          | & | &  & |
\end{bmatrix}.
\label{eq:x1andx2}
\end{equation}
If the snapshots satisfy \eqref{eq:approx},
then we have that
\begin{equation}
  \bm{X}_2 = \bm{A} \bm{X}_1 .
  \label{eq:A_approx}
\end{equation}
Otherwise, we hope to discover the ``best''
$\bm{A}$ which approximately satisfies this equation.

One solution to \eqref{eq:A_approx} is obtained by regression with least squares minimization; we define this solution to be $\hat{\bm{A}}$. 
In general, \eqref{eq:A_approx} may be consistent (having at least one solution), or inconsistent (having no solution). 
With these two cases, the corresponding minimization problem takes the form,
\begin{equation}  \bm{\hat{A}} = \begin{cases} 
      \min_{\bm{A}} \norm{\bm{A}}_F \text{ such that } \bm{X}_2 = \bm{A} \bm{X}_1 & \text{if \eqref{eq:A_approx} is consistent} \\
       \min_{\bm{A}} \norm{\bm{X}_2 - \bm{A} \bm{X}_1}_F & \text{if \eqref{eq:A_approx} is inconsistent}
   \end{cases}
   \label{eq:cases}.
\end{equation}
The solution in either case is given by the least squares fit, 
\begin{equation}
  \bm{\hat{A}} \coloneqq \bm{X}_2 \bm{X}_1^{\dagger},
  \label{eq:DMD_pseudoinverse}
\end{equation}
where $\bm{X}_1^{\dagger}$ denotes the Moore-Penrose pseudo-inverse of $\bm{X}_1$ \cite{penrose1955generalized}.
The DMD modes and eigenvalues in the SVD approach are the eigenvectors and eigenvalues of $\bm{\hat{A}}$, respectively. 

When $n$ is large, it may not be practical to compute $\hat{\bm{A}} \in \R^{n \times n}$ and its eigendecomposition directly. 
If $\bm{X}_1$ is low rank or approximately low rank, we may project the dynamics to a lower dimensional basis.  
In particular, if $\bm{X}_1$ has rank $r$, we may compute the reduced SVD,
\begin{equation}
  \bm{X}_1^r \coloneqq \bm{U}_r \bm{\Sigma}_r \bm{V}_r^{\intercal},
  \label{eq:x1svd}
\end{equation}
where the left singular vectors 
$\bm{U}_r \in \R^{n \times r}$ 
and right singular vectors 
$\bm{V}_r \in \R^{r \times T}$ 
are orthogonal matrices and 
$\bm{\Sigma}_r$ is a real positive diagonal matrix \cite{golub1971singular}. 
When measurement noise is present, 
we define $r$ to be the effective rank of the system
(discussed in detail in Section~\ref{subsec:effective_rank}).

We can then define the matrix,
\begin{equation*}
  \tilde{\bm{A}} \coloneqq \bm{U}_r^{\intercal} \bm{X}_2 \bm{V}_r \bm{\Sigma}_r^{-1},
\end{equation*}
where $\tilde{\bm{A}} \in \R^{r \times r}$ is much smaller in size than $\bm{A}$. 
Importantly, Tu \emph{et al.} showed that the 
eigenvalues of $\tilde{\bm{A}}$ 
are precisely the nonzero eigenvalues of $\bm{A}$~\cite{tu2013dynamic}.
The corresponding eigenvectors $\bm{\phi}_i$ of $\bm{A}$ can be found by first computing the eigenvectors $\bm{w}_i$ of $\tilde{\bm{A}}$,
\begin{equation*}
  \tilde{\bm{A}} \bm{w}_i = \lambda_i \bm{w}_i,
\end{equation*}
and then projecting into the original measurement space,
\begin{equation}
  \bm{\phi}_i = \frac{1}{\lambda_i} \bm{X}_2 \bm{V}_r \bm{\Sigma}_r^{-1} \bm{w}_i.
  \label{eq:exact_modes}
\end{equation}
In the case where the ranges of $\bm{X}_1$ and $\bm{X}_2$ are equal, \eqref{eq:exact_modes} reduces to $\bm{\phi}_i = \bm{U}_r \bm{w}_i$.

\subsection{Companion Matrix Approach} \label{subsec:companionmatrix}
An alternative formulation of DMD focuses on the computation of a
so-called companion matrix.
Although it is less commonly used in practice, this original formulation by Schmid~\cite{schmid2010dynamic} is analytically simpler and has been used in some key theoretical work~\cite{chen2012variants,arbabi2017ergodic}.

We again consider $T+1$ snapshots $\bm{x}_1, \ldots, \bm{x}_{T + 1} \in \R^n$ which satisfy \eqref{eq:approx}. 
We may express the last snapshot $T+1$ as a linear combination of the first $T$ states and a residual $\bm{r} \in \R^n$ which is orthogonal to these $T$ states,
\begin{equation*}
  \bm{x}_{T+1} = \sum_{j = 1}^{T} c_j \bm{x}_j + \bm{r} \text{ such that } \bm{r} \perp \text{span}\set{{\bm{x}_1, \ldots \bm{x}_{T}}},
\end{equation*}
where $c_j \in \R$.
Equivalently, we may write in matrix form,
\begin{equation}
  \bm{X}_2 = \bm{X}_1 \bm{C} + \bm{r} \bm{e}_{T}^{\intercal}
  \label{eq:companion_solution}
\end{equation}
where $\bm{e}_T = [0, \ldots, 0 , 1]^{\intercal}$, and
\begin{equation}
  \bm{C} = \begin{bmatrix} 
    0 & 0 &\cdots & 0 & c_1 \\
    1 & 0 & \cdots & 0 & c_2 \\
    0 & 1 & & 0 & c_3 \\
    \vdots & & \ddots & & \vdots \\
    0 & 0 & & 1 & c_T
   \end{bmatrix}
   \label{eq:companion_matrix}
\end{equation}
is called the companion matrix.
The least squares solution for 
$\bm{c} = [c_1, \ldots, c_T ]$ 
is then given by $\bm{c} =  \bm{X}_1^{\dagger} \bm{x}_{T+1}$. 
Note that all of the residual error in the model is placed on the last time snapshot.
The least squares solution $\bm{C}$ to \eqref{eq:companion_solution} is unique if and only if $\bm{x}_1, \ldots, \bm{x}_T$ are linearly independent \cite{chen2012variants}. 
If $\bm{x}_1, \ldots, \bm{x}_T$ are linearly independent, then $\bm{C}$ must also equal the least squares solution, 
\begin{equation*}
  \bm{C} = \bm{X}_1^{\dagger} \bm{X}_2.
\end{equation*} 
In some cases, 
the DMD modes (eigenvalues and eigenvectors of $\bm{A}$,
assuming \eqref{eq:approx})
are related to the eigenvalues and eigenvectors
of the companion matrix $\bm{C}$ \cite{schmid2010dynamic},
but these eigenvalues are, in general, not equal. 
In particular, the eigenvalues are only guaranteed to be equal if the columns of $\bm{X}_1$ are linearly independent \cite{chen2012variants,schmid2010dynamic}.

\subsection{Rank vs.\ Effective Rank}
\label{subsec:effective_rank}
If $\bm{X}_1$ has full column rank, then the companion matrix approach described in Section~\ref{subsec:companionmatrix} is equivalent to computing the DMD modes as in \eqref{eq:DMD_modes}. 
In the presence of measurement noise, $\bm{X}_1$ will almost surely have full column rank
even in the case where $\bm{A}$ is low-rank ($r < T < n$).
In that case, even though the companion matrix approach \eqref{eq:companion_matrix} has a well-posed solution, it yields the wrong number of eigenvalues.
Specifically, the companion matrix approach yields $T$ modes while there are only $r$ signal modes masked by noise. 
On the other hand, the SVD-based approach \eqref{eq:x1svd}
can filter out these noise modes
with a good estimate of $\mathrm{rank}(\bm{A})$.
Formally, we define the effective rank as follows:
\begin{definition}
Given a set of noisy measurements
$
  \bm{Y} = \bm{X} + \eta\bm{Z},
$
where $\bm{X}$ is low rank and  elements of $\bm{Z}$ are drawn independently from a random distribution with zero mean
and finite variance, we define the \textbf{effective rank} of $\bm{Y}$ to be the rank of $\bm{X}$. 
\end{definition}

In other words, the effective rank of $\bm{Y}$ is the rank of the data with no measurement noise ($\eta = 0$).
In general, the effective rank of the data is unknown. 
However, it may be estimated from the SVD spectrum~\cite{gavish2014optimal,wall2003singular}.
We now claim (and later show, in Section~\ref{sec:DFT}) that the companion matrix approach yields the DMD modes if and only if $\bm{X}_1$ not only has full column rank but also 
\textit{full effective column rank}. 
Although subtle, this distinction will play an important role in Section~\ref{sec:DFT}.

\section{Centering Data}
\label{sec:centering_data}


DMD as defined in \eqref{eq:x1andx2} and \eqref{eq:A_approx} can be thought of as a multivariate regression of the dynamics.
If the mean of $\bm{X}$ is not zero, as would occur with data measured about a non-zero equilibrium or data acquired over a short time interval, then the DMD model would be improved with an additional affine term:

\begin{equation}
	\bm{X}_2 = \bm{A} \bm{X}_1 + \bm{b} \bm{1}^{\intercal},
	\label{eq:affine}
\end{equation}
where $\bm{b} \in \R^n$ and $\bm{1}$ is a vector of length $T$ whose elements are all one. The corresponding minimization problem to find $\bm{A}$ and $\bm{b}$ is given by
\begin{equation}
  \tilde{\bm{A}}, \tilde{\bm{b}} = \begin{cases} \argmin_{\bm{A},\bm{b}} \norm{\bm{A}}_F \text{ s.t. } \bm{A} \bm{X}_1 + \bm{b} \bm{1}^{\intercal} = \bm{X}_2 & \text{if  \eqref{eq:affine} is consistent}\\
  \argmin_{\bm{A},\bm{b}} \norm{\bm{A} \bm{X}_1 + \bm{b} \bm{1}^{\intercal} - \bm{X}_2}_F^2 & \text{if \eqref{eq:affine} is inconsistent}.
  \end{cases}
  \label{eq:min_affine}
\end{equation}


As illustrated in Fig.~\ref{fig:motivation}, the incorporation of an affine term in the one-dimensional regression model is equivalent to centering $x_j$ and $y_j$ in the data.
For high-dimensional data, we compute the means of $\bm{X}_1$ and $\bm{X}_2$ as
\begin{equation*}
  \bm{\mu}_1 = \frac{\bm{X}_1 \bm{1}}{\bm{1}^{\intercal} \bm{1}}
  \quad \text{and} \quad
  \bm{\mu}_2 = \frac{\bm{X}_2\bm{1}}{\bm{1}^{\intercal} \bm{1}}.
\end{equation*}
The corresponding mean-subtracted or {\it centered}
data matrices are
\begin{equation*}
  \bar{\bm{X}}_1 = \bm{X}_1 - \bm{\mu}_1 \bm{1}^{\intercal}
  \quad \text{and} \quad
  \bar{\bm{X}}_2 = \bm{X}_2 - \bm{\mu}_2 \bm{1}^{\intercal}, 
\end{equation*} and we now solve the unbiased regression problem
\begin{equation}
  \bar{\bm{X}}_2 = \bar{\bm{A}} \bar{\bm{X}}_1.
  \label{eq:center}
\end{equation}
The least squares solution to \eqref{eq:center} is given by
\begin{equation}
  \bar{\bm{A}} = \begin{cases} \argmin_{\bm{A}} \norm{\bm{A}}_F \text{ s.t. } \bm{A}  \bar{\bm{X}}_1  =\bar{\bm{X}}_2 & 
  \text{if \eqref{eq:center} is consistent} \\
  \argmin_{\bm{A}} \norm{\bm{A} \bar{\bm{X}}_1  - \bar{\bm{X}}_2}_F^2 & 
  \text{if \eqref{eq:center} is inconsistent}
  \end{cases}
  \label{eq:min_center}.
\end{equation}

Importantly, the minimization problem \eqref{eq:min_center} is simpler to solve than the one in \eqref{eq:min_affine}.
We show in Proposition~\ref{prop:affine_center} that they are equivalent, yielding $\tilde{\bm{A}} = \bar{\bm{A}}$. 
The following Proposition, which we include for completeness,
is well-known among statisticians
in the setting of multivariate regression:



\begin{proposition}
Let
$\bm{X}_1$ and $\bm{X}_2 \in \R^{n \times T}$ 
be arbitrary matrices. 
The minimization problems \eqref{eq:min_affine} and \eqref{eq:min_center} are equivalent, 
with solutions $\tilde{\bm{A}} = \bar{\bm{A}}$ and $\tilde{\bm{b}} = \bm{\mu}_2 -  \bar{\bm{A}} \bm{\mu}_1$.
\label{prop:affine_center}
\end{proposition} 
\begin{proof}
We have two cases to consider, depending on whether the affine system of equations \eqref{eq:affine} is linearly consistent (has at least one solution) or inconsistent (has no solution). 
We will show that system \eqref{eq:affine} is consistent if and only if \eqref{eq:center} is consistent as well.

\begin{caseof}
\case{Consistent}
When \eqref{eq:affine} is consistent, 
the affine problem \eqref{eq:min_affine} is in the constrained 
(at least one solution for $\bm{A}$) case.
Note that we do not minimize over the norm of $\bm{b}$. 
Multiplying the constraint by $\frac{\bm{1}}{\bm{1}^{\intercal} \bm{1}}$ yields,
\begin{equation*}
  \bm{A} \frac{\bm{X}_1 \bm{1}}{\bm{1}^{\intercal} \bm{1}} + \bm{b} \frac{\bm{1}^{\intercal} \bm{1}}{\bm{1}^{\intercal} \bm{1}} = \frac{\bm{X}_2 \bm{1}}{\bm{1}^{\intercal} \bm{1}},
  \label{eq:affine_underdetermined}
\end{equation*}
which can be rearranged to find
$
  \tilde{\bm{b}} = \bm{\mu}_2 - \bm{A} \bm{\mu}_1.
$
Thus we can write \eqref{eq:min_affine} as
\begin{equation*}
  \min_{\bm{A}} \norm{\bm{A}}_F \text{ such that } \bar{\bm{X}}_2 = \bm{A} \bar{\bm{X}_1},
\end{equation*}
which is precisely \eqref{eq:min_center}.
Note that, since we assumed the constraint is satisfiable,
this implies that the centered system of equations \eqref{eq:center}
is consistent.

\case{Inconsistent}
If no solution to \eqref{eq:affine} exists, 
then we minimize the residual error without constraints.
Taking the gradient with respect to $\bm{b}$ and setting it equal to $0$ yields
\begin{equation*}
\bm{1}^{\intercal} \bm{X}_1^{T} \bm{A}^{\intercal} + \bm{b}^{\intercal} \bm{1}^{\intercal} \bm{1} = \bm{1}^{\intercal} \bm{X}_2
\end{equation*}
and rearranging,
we again find that
$
	\tilde{\bm{b}} = \bm{\mu}_2 - \bm{A} \bm{\mu}_1.
$
Plugging this into \eqref{eq:min_affine}, yields the minimization problem
\begin{equation*}
\begin{split}
  \min_{\bm{A}} \norm{\bm{A} \bm{X}_1 + \left( \bm{\mu}_2 - \bm{A} \bm{\mu}_1 \right) \bm{1}^{\intercal} - \bm{X}_2}_F^2  &= \min_{\bm{A}} \norm{\bm{A} \left( \bm{X}_1 - \bm{\mu}_1  \bm{1}^{\intercal} \right) - \left( \bm{X}_2 - \bm{\mu}_2 \bm{1}^{\intercal} \right) }_F^2 \\
  &= \min_{\bm{A}} \norm{\bm{A} \bar{\bm{X}}_1 - \bar{\bm{X}}_2  }_F^2,
\end{split}
\end{equation*}
which is precisely \eqref{eq:min_center}.
Note that this also must be inconsistent, otherwise the affine problem
would be consistent, and we would obtain a contradiction.
\end{caseof}
\end{proof}

\begin{remark} We make no assumptions about the matrices 
$\bm{X}_1$ and $\bm{X}_2$ in Proposition~\ref{prop:affine_center}. 
Therefore, this result does not depend on the system being linear or being generated by a dynamical system, and thus it is applicable in all regression settings.
\end{remark}

Instead of centering $\bm{X}_1$ and $\bm{X}_2$ individually, we may also choose to subtract the overall mean  $\bm{\mu} = \frac{1}{T + 1} \sum_{j = 1}^{T+1} \bm{x}_j$ from the data. Mean-subtraction of data and normalization of variance is standard in matrix factorization algorithms such as PCA~\cite{wold1987principal} and ICA~\cite{hyvarinen1999fast}. 
In many cases, $\bm{\mu}$ is very similar to $\bm{\mu}_1$ and $\bm{\mu}_2$. 
In particular, $\bm{\mu}_1, \bm{\mu}_2$, and $\bm{\mu}$ are all approximately equal in the case of neutral dynamics (all of the DMD eigenvalues lie near the unit circle). 
However, in the presence of transients or unstable behavior, these three values may be very different.

\section{Uniqueness of Modes}
\label{sec:Uniqueness}
The remainder of this paper compares DMD modes and eigenvalues computed with and without centering. To perform such a comparison, it is necessary that we first establish uniqueness of the DMD modes and corresponding eigenvalues for a linear system (Section \ref{subsec:uniqueness_dmd}). We then follow with a similar proof for the uniqueness of modes data generated by an affine linear system (Section \ref{subsec:uniqueness_affine}). This is key for showing that the modes from DMD with centering are well-defined. 

\subsection{Uniqueness of Dynamic Mode Decomposition}
\label{subsec:uniqueness_dmd}
Following \eqref{eq:approx}, \eqref{eq:x1andx2}, and \eqref{eq:A_approx}, assume we have sequential snapshots of data $\bm{x}_1, \ldots, \bm{x}_{T+1} \in \R^n$ that are generated by linear dynamics \eqref{eq:approx}.
In general, there may be infinitely many matrices $\bm{A}'$ that satisfy
\begin{equation*}
  \bm{x}_{j + 1} = \bm{A}' \bm{x}_j.
\end{equation*}
Chen \textit{et al.} show, using the companion matrix approach,
that although $\bm{A}'$ is not unique, 
the corresponding eigenvectors and eigenvalues of $\bm{A}'$ are:
\begin{theorem}[Chen et al.\ \cite{chen2012variants}, Theorem 1 (rephrased)]
The choice of eigenvalues $\lambda_1, \ldots, \lambda_n$ and corresponding eigenvectors $\bm{v}_1, \ldots, \bm{v}_n$ are unique up to a reordering in $j$, if and only if $\bm{x}_1, \ldots, \bm{x}_T$ are linearly independent and $\lambda_1, \ldots, \lambda_n$ are distinct.
\end{theorem}

In other words, even though $\bm{A}'$ is not unique, 
all $n$ eigenvalues and eigenvectors of $\bm{A}'$ 
are unique if and only if $\bm{X}$ has full column rank and the eigenvalues are distinct. In the case of low-rank data, $\bm{X}$ will not have full column rank and the eigenvalues of $\bm{A}'$ will not be distinct, since $\bm{A}'$ may have a zero eigenvalue with multiplicity greater than $1$. 
Consequently, this Theorem does not provide much relevant information about uniqueness in the case of low-rank dynamics. 
To remedy this, we generalize this result to the case of low-rank data and prove that the nonzero eigenvalues and corresponding eigenvectors are unique. 
We first establish two useful lemmas:
\begin{lemma}
Consider the $(p+1) \times q$ {\em rectangular} Vandermonde matrix
\begin{equation*}
  \bm{\Lambda} = \begin{bmatrix} 1 & 1 & \cdots & 1 \\
                                \lambda_1 & \lambda_2 & \cdots & \lambda_q \\
                                \vdots & \vdots & \vdots & \vdots \\
                                \lambda_1^p & \lambda_2^p & \cdots & \lambda_q^p
   \end{bmatrix} .
\end{equation*}
Then the $q$ columns of $\bm{\Lambda}$ are linearly independent 
($\bm{\Lambda}$ has full column rank) if and only if 
$q \leq p+1$ and $\lambda_1, \lambda_2, \ldots, \lambda_q$ are distinct.
\label{lemma:vandermonde}
\end{lemma}
\begin{proof}
Assume $q \leq p + 1$ (if not, $\mathrm{rank}(\bm{\Lambda}) \leq p + 1 < q$).
We form the $q \times q$ submatrix 
\begin{equation*}
\begin{bmatrix} 1 & 1 & \cdots & 1 \\
                                \lambda_1 & \lambda_2 & \cdots & \lambda_r \\
                                \vdots & \vdots & \vdots & \vdots \\
                                \lambda_1^{q-1} & \lambda_2^{q-1} & \cdots & \lambda_q^{q-1}
   \end{bmatrix},
\end{equation*}
which has nonzero determinant \cite{turner1966inverse} if and only if the eigenvalues are distinct.
\end{proof}
\begin{lemma} Suppose we have sequential time series snapshots $\bm{x}_1, \ldots, \bm{x}_{T+1}$  such that $\bm{x}_{j+1} = \bm{A} \bm{x}_j$ for $j = 1, \ldots, T$, where $\bm{A}$ is diagonalizable.
\begin{itemize}
  \item If $\mathrm{range}(\bm{X}_1) = \mathrm{range}(\bm{X}_2)$, then $\bm{X}$ may be expressed as
  \begin{equation}
    \bm{X} = \underbrace{\begin{bmatrix} 
  | & | &  & | \\
  \bm{v}_1 & \bm{v}_2 & \cdots & \bm{v}_r \\
  | & | & & |
  \end{bmatrix}}_{\bm{V}} \underbrace{\begin{bmatrix}
  1 & \lambda_1 & \lambda_1^2 & \cdots & \lambda_1^{T} \\
  1 & \lambda_2 & \lambda_2^2 & \cdots & \lambda_2^{T} \\
  \vdots & \vdots & \vdots & \cdots & \vdots \\
  1 & \lambda_r & \lambda_r^2 & \cdots & \lambda_r^{T}
  \end{bmatrix}}_{\bm{\Lambda}^{\intercal}},
  \label{eq:decomp1}
  \end{equation}
  where $\lambda_1, \ldots, \lambda_r$ and  $\bm{v}_1, \ldots \bm{v}_r$ are distinct nonzero eigenvalues and eigenvectors of $\bm{A}$, respectively.

  \item If $\mathrm{range}(\bm{X}_1) \neq \mathrm{range}(\bm{X}_2)$, then $\bm{X}$ may be expressed as 
  \begin{equation}
      \bm{X} = \underbrace{\begin{bmatrix} 
  | & | & | &  & | \\
  \bm{v}_0 & \bm{v}_1 & \bm{v}_2 & \cdots & \bm{v}_r \\
  |& | & | & & |
  \end{bmatrix}}_{\bm{V}} \underbrace{\begin{bmatrix}
  1 & 0 & 0 & \cdots & 0 \\
  1 & \lambda_1 & \lambda_1^2 & \cdots & \lambda_1^{T} \\
  1 & \lambda_2 & \lambda_2^2 & \cdots & \lambda_2^{T} \\
  \vdots & \vdots & \vdots & \cdots & \vdots \\
  1 & \lambda_r & \lambda_r^2 & \cdots & \lambda_r^{T}
  \end{bmatrix}}_{\bm{\Lambda}^{\intercal}},
  \label{eq:decomp2}
  \end{equation}
  where $\bm{v}_0 \in \mathrm{Null}(\bm{A})$.
\end{itemize}

\label{lemma:factor}
\end{lemma}
\begin{proof}
Assume $\text{range}(\bm{X}_1) = \text{range}(\bm{X}_2)$. Since $\bm{X}_2 = \bm{A} \bm{X}_1$, and $\text{range}(\bm{X}_1) = \text{range}(\bm{X}_2)$, then $\text{range}(\bm{X}_2) \subseteq \text{range}(\bm{A})$. Since $\bm{A}$ is diagonalizable, then we express $\bm{x}_1$ as a linear combination of the nonzero eigenvectors of $\bm{A}$, scaled appropriately, so that
\begin{equation*}
  \bm{x}_1 = \sum_{i = 1}^r \bm{v}_i.
\end{equation*}
We note that eigenvalues corresponding to the $\bm{v}_i$'s are distinct. Otherwise, they can be summed together in the initial condition $\bm{x}_1$.
Recursively applying $\bm{A}$ to $\bm{x}_1$,
\begin{gather*}
  \bm{x}_2 = \bm{A} \bm{x}_1 = \bm{A} \sum_{i = 1}^r \bm{v}_i = \sum_{i = 1}^r \lambda_i \bm{v}_r, \\
  \bm{x}_3 = \bm{A}^2 \bm{x}_1 = \sum_{i = 1}^r \lambda_i^2 \bm{v}_i, 
\end{gather*}
and in general,
\begin{equation*}
  \bm{x}_k = \bm{A}^{k - 1} \bm{x}_1 = \sum_{i = 1}^r \lambda_i^{k - 1} \bm{v}_i.
\end{equation*}
Putting this in matrix form yields \eqref{eq:decomp1}.

If $\text{range}(\bm{X}_1) \neq \text{range}(\bm{X}_2)$, then $\bm{x}_1$ has a component $\bm{v}_0$ which lies in $\text{Null}(\bm{A})$. We scale $\bm{v}_0$ so that $\bm{x}_1 = \sum_{i = 0}^r \bm{v}_i$.
Since $\bm{v}_0$ has an associated eigenvalue $\lambda_0 = 0$, for $k \geq 1$,
\begin{equation*}
  \bm{x}_k = \bm{A}^{k - 1} \bm{x}_1 = \sum_{i = 0}^r \lambda_i^{k - 1} \bm{v}_i = \sum_{i = 1}^r \lambda_i^{k - 1} \bm{v}_i, 
\end{equation*}
which yields the decomposition in \eqref{eq:decomp2}.
\end{proof}
Note that since the columns of $\bm{V}$ correspond to eigenvectors of distinct eigenvalues of $\bm{A}$, $\bm{V}$ has full column rank.
We now introduce a definition of well-posedness for the 
DMD problem.

\begin{definition} 
Suppose we have sequential time series snapshots $\bm{x}_1, \ldots, \bm{x}_{T+1}$ such that $\bm{x}_{j+1} = \bm{A} \bm{x}_j$. Let $\bm{A}$ have $r$ nonzero and distinct eigenvalues
$\lambda_1, \ldots, \lambda_r$
and corresponding eigenvectors 
$\bm{v}_1, \ldots, \bm{v}_r$.
We say that the {\bf DMD problem is well-posed}
if the conditions
\begin{enumerate}
 \item $\bm{x}_1$ is not orthogonal to any $\bm{v}_1, \ldots, \bm{v}_r$, and either
  \item $T \geq r$ and $\bm{X}_1$ and $\bm{X}_2$ share the same range, or
  \item $T \geq r+1$,
\end{enumerate}
are satisfied. 
\end{definition}

Now we prove our main uniqueness theorem.
\begin{theorem}[Uniqueness of Dynamic Mode Decomposition] Suppose we have sequential time series snapshots $\bm{x}_1, \ldots, \bm{x}_{T+1}$ 
such that $\bm{x}_{j+1} = \bm{A} \bm{x}_j$ for $j = 1,\ldots,T$, 
where $\bm{A}$ has $r$ nonzero and distinct eigenvalues $\lambda_1, \ldots, \lambda_r$ and corresponding eigenvectors, $\bm{v}_1, \ldots, \bm{v}_r$. 
Let $\bm{A}'$ be any other rank $r$ matrix which satisfies
$\bm{x}_{j+1} = \bm{A}' \bm{x}_j$. 
If the DMD problem is well-posed,
then $\bm{A}'$ has the same $r$ nonzero eigenvalues $\lambda_1, \ldots, \lambda_r$ and corresponding eigenvectors $\bm{v}_1, \ldots, \bm{v}_r$ as $\bm{A}$, and these are unique up to scaling.
\label{theorem:uniqueness}
\end{theorem}
\begin{proof} First, suppose $\text{range}(\bm{X}_1) = \text{range}(\bm{X}_2)$. Since $\bm{x}_{j+1} = \bm{A} \bm{x}_j$, if we define $\bm{X}$ to be the matrix containing all $T+1$ snapshots, then by \ref{lemma:factor} we can factor $\bm{X}$ as follows,
\begin{equation}
\bm{X} = \begin{bmatrix} 
| & | & & | \\
\bm{x}_1 & \bm{x}_2 & \cdots & \bm{x}_{T+1} \\
| & | & & |
\end{bmatrix} = 
\underbrace{\begin{bmatrix} 
  | & | &  & | \\
  \bm{v}_1 & \bm{v}_2 & \cdots & \bm{v}_r \\
  | & | & & |
  \end{bmatrix}}_{\bm{V}} \underbrace{\begin{bmatrix}
  1 & \lambda_1 & \lambda_1^2 & \cdots & \lambda_1^{T} \\
  1 & \lambda_2 & \lambda_2^2 & \cdots & \lambda_2^{T} \\
  \vdots & \vdots & \vdots & \cdots & \vdots \\
  1 & \lambda_r & \lambda_r^2 & \cdots & \lambda_r^{T}
  \end{bmatrix}}_{\bm{\Lambda}^{T}},
\label{eq:Xfactorization}
\end{equation}
where $\bm{v}_j$ are the eigenvectors of $\bm{A}$ scaled appropriately so that $\bm{x}_1 = \sum_{i = 1}^r \bm{v}_j$. We denote these matrices $\bm{V}$ and $\bm{\Lambda}$ and note that $\bm{V}$ and $\bm{\Lambda}$ have full column rank
(Lemma~\ref{lemma:vandermonde}).

Suppose there exists another solution with corresponding eigenvalues $\lambda_1', \ldots \lambda'_r$ and eigenvectors $\bm{v}'_1, \ldots, \bm{v}'_r$. We construct another factorization,
\begin{equation*}
\bm{X} = \begin{bmatrix} 
  | & | &  & | \\
  \bm{v}'_1 & \bm{v}'_2 & \cdots & \bm{v}'_r \\
  | & | & & |
  \end{bmatrix} \begin{bmatrix}
  1 & \lambda_1' & \lambda_1^{'2} & \cdots & \lambda_1^{'T} \\
  1 & \lambda_2' & \lambda_2^{'2} & \cdots & \lambda_2^{'T} \\
  \vdots & \vdots & \vdots & \cdots & \vdots \\
  1 & \lambda_r' & \lambda_r^{'2} & \cdots & \lambda_r^{'T}
  \end{bmatrix}.
\end{equation*}
For the eigenvalues of $\bm{A}$ and $\bm{A}'$ to be different, there must exist some $\lambda_i'$ which does not equal $\lambda_1, \ldots, \lambda_r$. Since $\left\{[1, \ \lambda_i, \ \lambda_i^2, \cdots, \ \lambda_i^T] \text{ for } i = 1, \ldots, r \right\}$ spans the row space of $\bm{X}_2$, then $[1 \text{ } \lambda_i' \text{ } \lambda_i^{'2} \cdots \lambda_i^{'T}]$ must lie in this row space. Hence the $(r + 1) \times (T + 1)$ matrix,
\begin{equation}
\begin{bmatrix}
  1 & \lambda_1 & \lambda_1^2 & \cdots & \lambda_1^T \\
  1 & \lambda_2 & \lambda_2^2 & \cdots & \lambda_2^{T} \\
  \vdots & \vdots & \vdots & \cdots & \vdots \\
  1 & \lambda_r & \lambda_r^2 & \cdots & \lambda_r^{T} \\
  1 & \lambda_i' & \lambda_i^{'2} & \cdots & \lambda_i^{'T}
\end{bmatrix}
\label{eq:vandermonde}
\end{equation}
must have low row rank. However,since $r + 1 \leq T+1$ by assumption, and all the $\lambda_j$'s and $\lambda_i'$ are all distinct, by Lemma \ref{lemma:vandermonde} \eqref{eq:vandermonde} must have full row rank. With this contradiction we conclude that the nonzero eigenvalues of $\bm{A}$ are unique. 

For the eigenvectors of $\bm{A}$, note that $\bm{X} = \bm{V \Lambda }^{\intercal}$. Since $\bm{\Lambda}$ has full column-rank there is a unique solution for $\bm{V} = \bm{X} \bm{\Lambda}^{\intercal^{\dagger}}$. Thus, the eigenvectors of $\bm{A}$ must be unique up to a scaling. Note that since $\bm{X}$ and $\bm{\Lambda}$ have rank $r$, $\bm{V}$ must also have rank $r$ and thus have full column rank.

Note that we assumed that 
$\bm{X}_1$ and $\bm{X}_2$ share the same range. 
If they do not, since $\bm{X}_2 = \bm{A} \bm{X}_1$, there must be a component of $\bm{x}_1$ which is in the nullspace of $\bm{A}$. This results in appending an extra column $\bm{v_0}$, which is in the nullspace of $\bm{A}$, to $\bm{V}$ and an extra row $[1 \text{ } 0 \cdots 0]$ to $\bm{\Lambda}$ as in \eqref{eq:decomp2}. Following the same method, we find that $T \geq r$ must be replaced with $T \geq r+1$. 
\end{proof}

\subsection{Uniqueness of Affine Linear Model}
\label{subsec:uniqueness_affine}
As we have seen in Section \ref{sec:centering_data}, DMD with centering is equivalent to an additional affine term. Thus, in addition to proving that the DMD modes are unique it also important to show uniqueness of the modes for an affine dynamical system of the form $\bm{x}_{j+1} = \bm{A} \bm{x}_j + \bm{b}$. In the following Theorem, we assume that the matrix $\bm{A}$ does not contain an eigenvalue equal to $1$. If $1$ is an eigenvalue of $\bm{A}$, then there is an inherent ambiguity in whether this mode is an eigenvector of $\bm{A}$ or incorporated into $\bm{b}$.
First, we again define some conditions for the problem to
be well-posed:
\begin{definition} 
Suppose we have sequential time series snapshots $\bm{x}_1, \ldots, \bm{x}_{T+1}$ such that $\bm{x}_{j+1} = \bm{A} \bm{x}_j + \bm{b}$. Let $\bm{A}$ have $r$ nonzero and distinct eigenvalues, $\lambda_1, \ldots, \lambda_r$ and corresponding eigenvectors 
$\bm{v}_1, \ldots, \bm{v}_r$.
We say that the {\bf affine DMD problem is well-posed}
if the conditions
\begin{enumerate}
  \item $\bm{A}$ does not have an eigenvalue equal to $1$,
\item $\bm{x}_1 - \bm{c}$ is not orthogonal to $\bm{v}_1, \ldots, \bm{v}_r$, and either
  \item $T \geq r+1$ and $\bm{X}_1 - \bm{c} \bm{1}^{\intercal}$ and $\bm{X}_2 - \bm{c} \bm{1}^{\intercal}$ share the same range, or
  \item $T \geq r+2$,
\end{enumerate}
are satisfied, where $\bm{c} = \left( \bm{I} - \bm{A} \right)^{-1} \bm{b}$. 
\end{definition}

\begin{theorem}[Uniqueness of Affine DMD]
\label{theorem:rank_one_update}
Suppose we have sequential time series snapshots $\bm{x}_1, \ldots, \bm{x}_{T+1}$ such that $\bm{x}_{j+1} = \bm{A} \bm{x}_j + \bm{b}$ for $j = 1, \ldots, T$, where $\bm{A}$ has $r$ nonzero and distinct eigenvalues, $\lambda_1, \ldots, \lambda_r$ and corresponding eigenvectors $\bm{v}_1, \ldots, \bm{v}_r$. 
Let $\bm{A}'$ and $\bm{b}'$ be any other rank $r$ matrix and vector which satisfy $\bm{x}_{j+1} = \bm{A}' \bm{x}_j + \bm{b}'$. If the affine DMD problem is well-posed,
then $\bm{b}' = \bm{b}$ and $\bm{A}'$ has the same $r$ nonzero eigenvalues $\lambda_1, \ldots, \lambda_r$ and corresponding eigenvectors $\bm{v}_1, \ldots, \bm{v}_r$ as $\bm{A}$,
and these are unique up to scaling.
\end{theorem}
\begin{proof}
Since $\bm{A}$ does not have an eigenvalue of $1$, then $\bm{I} - \bm{A}$ is invertible. 
Therefore, we can shift the origin in order to express 
$\bm{x}_{j+1} = \bm{A} \bm{x}_j + \bm{b}$ 
as 
$\bm{x}_{j+1} - \bm{c} = \bm{A} ( \bm{x}_j - \bm{c} )$, 
where $\bm{c} = ( \bm{I} - \bm{A} )^{-1} \bm{b}$. 
By Lemma \ref{lemma:factor},
we may express $\bm{X} - \bm{c} \bm{1}^{\intercal}$ as
\begin{equation}
    \bm{X} - \bm{c} \bm{1}^{\intercal} = \bm{V} \bm{\Lambda}^{\intercal}.
    \label{eq:shift_origin_factor}
\end{equation}
Similarly for $\bm{A}'$ and $\bm{b}'$, $\bm{X} - \bm{c}' \bm{1}^{\intercal} = \bm{V}' \bm{\Lambda}'^{\intercal}$. Taking the difference in these equations,
\begin{equation*}
    \bm{V} \bm{\Lambda}^{\intercal} = \bm{V}' \bm{\Lambda}'^{\intercal} + (\bm{c} - \bm{c}') \bm{1}^{\intercal}.
\end{equation*}
First, assume that both $\bm{c} \neq \bm{c}'$ and that $\bm{A}$ and $\bm{A}'$ do not share all of their eigenvalues. 
(We will show that this yields a contradiction.)
Without loss of generality, let $\lambda$ be an eigenvalue of $\bm{A}$ but not $\bm{A}'$. Hence, $\bm{\lambda} = \left[ 1 \mbox{ } \lambda \mbox{ } \cdots \mbox{ } \lambda^T \right]^{\intercal}$ is a column of $\bm{\Lambda}$ but not $\bm{\Lambda}'$. Definining, $\tilde{\bm{\Lambda}} = [ \bm{\Lambda}' ~ \bm{1} ]$, then
\begin{equation*}
    \bm{V} \bm{\Lambda}^T = \begin{bmatrix} \bm{V}' & \bm{c} - \bm{c}'  \end{bmatrix} \begin{bmatrix} \bm{\Lambda}'^{\intercal} \\  \bm{1}^{\intercal} \end{bmatrix} = \begin{bmatrix} \bm{V}' & \bm{c} - \bm{c}'  \end{bmatrix} \tilde{\bm{\Lambda}}^{\intercal}.
\end{equation*} Since $\bm{V}$ has full column rank, applying $\bm{V}^{\dagger}$ on the left to both sides yields,
\begin{equation*}
      \bm{\Lambda}^{\intercal} = \bm{V}^{\dagger} \begin{bmatrix} \bm{V}' & \bm{c} - \bm{c}'  \end{bmatrix} \tilde{\bm{\Lambda}}^{\intercal}.
\end{equation*}
If we apply the orthogonal projection $\bm{I} - \tilde{\bm{\Lambda}}^{\intercal^\dagger} \tilde{\bm{\Lambda}}^{\intercal}$ on the right to both sides, we see that the right hand side is $\bm{0}$.
However, $\bm{\Lambda}^{\intercal} \left( \bm{I} - \tilde{\bm{\Lambda}}^{\intercal^\dagger} \tilde{\bm{\Lambda}}^{\intercal} \right)$ cannot be $\bm{0}$ since, by assumption, $\bm{\Lambda}$ and $\bm{\Lambda}'$ have different column spaces. To see this, consider the solution $\bm{Y}$ to $\tilde{\bm{\Lambda}} \bm{Y} = \bm{\Lambda} $. In particular, consider a single column of this equation:
\begin{equation}
    \tilde{\bm{\Lambda}} \bm{y} = \bm{\lambda}.
\label{eq:lambda_column}
\end{equation}
If $\bm{x}_1 - \bm{c}$ is in ${\rm range}(\bm{A})$, then we have
\begin{equation*}
\underbrace{
    \begin{bmatrix}
    1 & 1 & \cdots & 1 & 1 \\
    \lambda_1' & \lambda_2' & \cdots & \lambda_r' & 1 \\    
    {\lambda_1'}^2 & {\lambda_2'}^2 & \cdots & {\lambda_r^2}' & 1 \\
    \vdots & \vdots & \vdots & \vdots & \vdots \\
    {\lambda_1'}^T & {\lambda_2'}^T & \cdots & {\lambda_r'}^T & 1 
    \end{bmatrix}}_{\tilde{\bm{\Lambda}}} \bm{y}
    = 
    \begin{bmatrix}
    1 \\ \lambda \\ \lambda^2 \\ \vdots \\ \lambda^T
    \end{bmatrix}.
\end{equation*}
By Lemma \ref{lemma:vandermonde}, if $T+1 \geq r+2$, then since $\lambda, \lambda_1', \ldots, \lambda_r', 1$ are all distinct, $\bm{\lambda}$ cannot be expressed as a linear combination of the columns of $\tilde{\bm{\Lambda}}^{\intercal}$. 
(If $\bm{x}_1 - \bm{c} \notin {\rm range}(\bm{A})$, 
then we must append an extra column to $\tilde{\bm{\Lambda}}$, 
and the condition in this case is $T+1 \geq r+3$).
Thus, there does not exist a solution for $\bm{y}$ in \eqref{eq:lambda_column}, and consequently there does not exist a solution for $\bm{Y}$. Thus, $\tilde{\bm{\Lambda}} \tilde{\bm{\Lambda}}^{\dagger} \bm{\Lambda} \neq \bm{\Lambda}$. Taking the transpose and rearranging we have
\begin{equation*}
    \bm{\Lambda}^{\intercal} \left( \bm{I} - \tilde{\bm{\Lambda}}^{\intercal^{\dagger}} \tilde{\bm{\Lambda}}^{\intercal} \right) \neq \bm{0}.
\end{equation*}
This yields a contradiction, from which we conclude that either $\bm{c} = \bm{c}'$, or $\bm{A}$ and $\bm{A}'$ have the same nonzero eigenvalues, or both.

First, consider the case where $\bm{c} = \bm{c}'$. This implies that $\bm{V}' \bm{\Lambda}'^{\intercal} = \bm{V} \bm{\Lambda}^{\intercal}$, and using the same logic as Theorem \ref{theorem:uniqueness}, then the nonzero eigenvalues of $\bm{A'}$ and $\bm{A}$ are equal.

Next, suppose that the nonzero eigenvalues of $\bm{A}$ and $\bm{A}'$ are equal.  Thus, $\bm{\Lambda} = \bm{\Lambda}'$ and
\begin{equation*}
    \left( \bm{V} - \bm{V}'  \right) \bm{\Lambda}^{\intercal} = \left( \bm{c} - \bm{c}' \right) \bm{1}^{\intercal}.
\end{equation*}
Applying $\bm{I} - \bm{\Lambda}^{\intercal^{\dagger}} \bm{\Lambda}^{\intercal}$ to both sides,
\begin{equation*}
    \bm{0} = \left( \bm{c} - \bm{c}' \right) \bm{1}^{\intercal} \left( \bm{I} - \bm{\Lambda}^{\intercal^{\dagger}} \bm{\Lambda}^{\intercal} \right),
\end{equation*}
and note that the right hand side is an $n \times (T+1)$ 
rank-1 matrix.
Now, $\bm{1}^{\intercal} \left( \bm{I} - \bm{\Lambda}^{\intercal^{\dagger}} \bm{\Lambda}^{\intercal} \right) \neq \bm{0}$, since $\bm{1}$ is not in the column space of $\bm{\Lambda}$. Thus, $\bm{c} = \bm{c}'$.

In either case we have that both $\bm{c} = \bm{c}'$ and $\bm{\Lambda}' = \bm{\Lambda}$. From $\bm{X} = \bm{V} \bm{\Lambda}^{\intercal} + \bm{c} \bm{1}^{\intercal}$, we see that $\bm{V}$ is the unique solution, $\bm{V} = \left( \bm{X} - \bm{c} \bm{1}^{\intercal} \right) \bm{\Lambda}^{\intercal^{\dagger}}$.
\end{proof}

\section{Comparison of DMD with Centering to DMD without Centering}
\label{sec:comparing_without_centering}
In this section, we show that, for linear systems (dynamics generated by $\bm{x}_{j+1} = \bm{A} \bm{x}_j$), both DMD with centering and without centering can be used to compute the modes of $\bm{A}$. 
In particular, DMD with and without centering will yield the same modes, except for the background mode. For DMD without centering, this background mode corresponds to an eigenvalue equal to $1$, while for DMD with centering, this is replaced by an eigenvalue equal to $0$ (see Theorem \ref{theorem:one_rank_update_1}).

For affine systems (dynamics generated by $\bm{x}_{j+1} = \bm{A} \bm{x}_j + \bm{b}$), DMD with centering can be used to extract $\bm{b}$ and the modes of $\bm{A}$. In some cases, DMD with centering can also be used to compute the modes of $\bm{A}$ and model the dynamics of the system, but in many cases it cannot,
most notably if $\bm{A}$ is full-rank. 
Here we provide necessary and sufficient conditions
for when DMD will and will not be able to successfully model the dynamics.

We then illustrate these results with synthetic examples in Section \ref{subsec:synthetic_examples}, and show that these results generalize to the case of measurement noise in Section \ref{subsec:noise}.

\subsection{Linear Systems without Bias}
\label{subsec:DMD_equivalent_without_bias}
Consider a set of snapshots $\bm{x}_j$ which satisfy $\bm{x}_{j+1} = \bm{A} \bm{x}_j$. From Theorem \ref{theorem:uniqueness}, we know that DMD without centering can be used to extract the nonzero eigenvalues and eigenvectors of $\bm{A}$. We now show that DMD with centering can also be used to compute the same modes. In particular, if one is not an eigenvalue of $\bm{A}$, DMD with and without centering yield the same DMD modes (Theorem \ref{theorem:one_rank_update_1}). If one is an eigenvalue of $\bm{A}$, then DMD with and without centering will share the same modes, except for the background mode. For DMD without centering this mode corresponds to an eigenvalue equal to one, while for DMD with centering this is replaced with an eigenvalue equal to zero.

We first prove a useful lemma.
\begin{lemma}
Suppose we have sequential time series such that $\bm{x}_{j+1} = \bm{A} \bm{x}_j$, and the DMD problem is well-posed. Then $\bm{A}$ has an eigenvalue equal to $1$ if and only if $\bm{1}^{\intercal} \bm{X}_1^{\dagger} \bm{X} = \bm{1}^{\intercal}$.
\label{lemma:mystery}
\end{lemma}

\begin{proof}
Let $\bm{X}_1$ have rank $r$.\footnote{In general, we define $r$ to be the rank of $\bm{A}$. In many cases, 
$\text{rank}(\bm{A}) = \text{rank}(\bm{X}_1)$. 
However, if $\bm{x}_1$ has a component in the nullspace of $\bm{A}$, then $\text{rank}(\bm{X}_1) = \text{rank}(\bm{A}) + 1$ and one of the $\lambda_j$'s will be $0$.} By Lemma \ref{lemma:factor}, we may decompose $\bm{X}$ into the product of two matrices $\bm{V}$ and $\bm{\Lambda}$ which have full column rank.

\begin{equation}
  \bm{X}_1 = \underbrace{\begin{bmatrix} 
  | & | &  & | \\
  \bm{v}_1 & \bm{v}_2 & \cdots & \bm{v}_r \\
  | & | & & |
  \end{bmatrix}}_{\bm{V}}
\underbrace{
  \begin{bmatrix}
  1 & \lambda_1 & \lambda_1^2 & \cdots & \lambda_r^{T-1} \\
  1 & \lambda_2 & \lambda_2^2 & \cdots & \lambda_2^{T-1} \\
  \vdots & \vdots & \vdots & \cdots & \vdots \\
  1 & \lambda_r & \lambda_r^2 & \cdots & \lambda_r^{T-1}
  \end{bmatrix}}_{\bm{\Lambda}^{\intercal}}
  \label{eq:x1factor}
\end{equation}
Thus, $\bm{V}^{\dagger} \bm{V} = \bm{I}$ and
\begin{equation*}
  \bm{X}_1^{\dagger} \bm{X}_1 = \bm{\Lambda}^{\dagger^{\intercal}} \bm{V}^{\dagger} \bm{V} \bm{\Lambda}^{\intercal} = \bm{\Lambda}^{\dagger^{\intercal}} \bm{\Lambda}^{\intercal} = \bm{\Lambda} (\bm{\Lambda}^{\intercal} \bm{\Lambda})^{-1} \bm{\Lambda}^{\intercal} = \bm{\Lambda} \bm{\Lambda}^{\dagger}.
\end{equation*}

First, suppose that $1$ is an eigenvalue of $\bm{A}$. Without loss of generality, let $\lambda_1 = 1$. Consider the solution $\bm{\beta}$ to the equation $\bm{\Lambda} \bm{\beta} = \bm{1}$, or more explicitly,
\begin{equation}
\begin{bmatrix}
  1 & 1 & \cdots & 1 \\
  1 & \lambda_2 & \cdots & \lambda_r \\ 
  1 & \lambda_2^2 & \cdots & \lambda_r^2  \\ 
  \vdots & \vdots & \cdots & \vdots \\
  1 & \lambda_2^{T-1} & \cdots & \lambda_k^{T-1} \\
  \end{bmatrix} \bm{\beta} = \begin{bmatrix} 1 \\ 1 \\ \vdots  \\ 1 \end{bmatrix}
  \label{eq:beta}
  \end{equation}
Clearly, there exists a solution for $\bm{\beta}$, namely $\bm{\beta} = [1 \text{ } 0 \cdots 0]^{\intercal}$. Since the columns of $\bm{\Lambda}$ are linearly independent, $\bm{\beta}$ is unique and hence $\bm{\beta} = \bm{\Lambda}^{\dagger} \bm{1} = [1 \text{ } 0 \cdots 0]^{\intercal}$ and
$\bm{X}_1^{\dagger} \bm{X}_1 \bm{1} = \bm{\Lambda} \bm{\Lambda}^{\dagger} \bm{1} = \bm{1}$. Since $\bm{X}_1^{\dagger} \bm{X}_1$ is symmetric $\bm{1}^{\intercal} = \bm{1}^{\intercal} \bm{X}_1^{\dagger} \bm{X}_1$.
To conclude the proof for this direction we must show that $\bm{1}^{\intercal} \bm{X}_1^{\dagger} \bm{x}_{T+1} = 1$.
We can express $\bm{x}_{T+1}$ as 
\begin{equation*}
\bm{x}_{T+1} = \begin{bmatrix} 
 | & | &  & | \\
 \bm{v}_1 & \bm{v}_2 & \cdots & \bm{v}_r \\
 | & | & & |
 \end{bmatrix} \underbrace{\begin{bmatrix} 1 \\ \lambda_2^T \\ \vdots \\ \lambda_r^T
 \end{bmatrix}}_{\bm{\lambda}_{T}}.
\end{equation*}
Plugging this in,
\begin{equation*}
 \bm{1}^{\intercal} \bm{X}_1^{\dagger} \bm{x}_{T+1} = \bm{1}^{\intercal} \bm{\Lambda}^{\dagger^{\intercal}} \bm{V}^{\dagger} \bm{V} \bm{\lambda}_{T} = \bm{1}^{\intercal} \bm{\Lambda}^{\dagger^{\intercal}} \bm{\lambda}_{T} = \bm{\beta}^{\intercal} \bm{\lambda}_{T+1} = 1.
\end{equation*}

Now suppose $1$ is not an eigenvalue of $\bm{A}$. Similar to \eqref{eq:beta}, we consider the equation, 
\begin{equation*}
\begin{bmatrix}
  1 & 1 & \cdots & 1 \\
  \lambda_1 & \lambda_2 & \cdots & \lambda_r \\ 
  \lambda_1^2 & \lambda_2^2 & \cdots & \lambda_r^2  \\ 
  \vdots & \vdots & \cdots & \vdots \\
  \lambda_1^{T-1} & \lambda_2^{T-1} & \cdots & \lambda_r^{T-1} \\
  \end{bmatrix} \bm{\beta} = \begin{bmatrix} 1 \\ 1 \\ \vdots  \\ 1 
  \end{bmatrix}.
  \end{equation*}
  In this case, all of the $\lambda_i$'s are distinct and not equal to $1$. By Lemma \ref{lemma:vandermonde}, since $T-1 \geq r$, $\bm{1}$ is not in the span of the columns of $\bm{\Lambda}$ and hence $\bm{\Lambda} \bm{\beta} \neq \bm{1}$ for any value of $\bm{\beta}$. Thus, $\bm{\Lambda} \bm{\Lambda}^{\dagger} \bm{1} \neq \bm{1}$ and therefore $\bm{1}^{\intercal} \bm{X}_1 \bm{X}_1^{\dagger} \neq \bm{1}^{\intercal}$.
  \end{proof}
\begin{theorem}[DMD with and without Centering for Linear Systems]
Suppose we have sequential time series snapshots $\bm{x}_1, \ldots, \bm{x}_{T+1}$ such that $\bm{x}_{j+1} = \bm{A} \bm{x}_j$ for $j = 1, \ldots, T$. If the DMD problem is well-posed, then the following holds for $\hat{\bm{A}} = \bm{X}_2 \bm{X}_1^{\dagger}$: 
\begin{enumerate}
    \item If $\hat{\bm{A}}$ has an eigenvalue equal to $1$, then $\bar{\bm{A}} = \bar{\bm{X}}_2 \bar{\bm{X}}_1^{\dagger}$ will have the same eigenvalues and corresponding eigenvectors as $\hat{\bm{A}}$, except the $1$ eigenvalue is replaced with a $0$ eigenvalue.
    \item If $\hat{\bm{A}}$ does not have an eigenvalue equal to $1$, then $\bar{\bm{A}}$ will have the same eigenvalues and corresponding eigenvectors as $\hat{\bm{A}}$.
\end{enumerate}
\label{theorem:one_rank_update_1}
\end{theorem}
\begin{proof}
DMD with centering obtains the centered matrix
\begin{equation}
\begin{split}
\bar{\bm{A}} &= \bar{\bm{X}}_2 \bar{\bm{X}}_1^{\dagger} = (\bm{X}_2 - \bm{\mu}_2 \bm{1}^{\intercal}) (\bm{X}_1 - \bm{\mu}_1 \bm{1}^{\intercal})^{\dagger}. 
\end{split}
\label{eq:rank_one_update}
\end{equation}
For part 1, suppose that $\hat{\bm{A}}$ has an eigenvalue equal to $1$. 
Then we have that
$(\bm{I} - \bm{X}_1^{\dagger} \bm{X}_1)^{\intercal} \bm{1} = \bm{0}$ 
by Lemma \ref{lemma:mystery}.
Applying the rank one update formula in \cite{petersen2008matrix}, which is a generalization of the Sherman-Morrison-Woodbury formula \cite{sherman1950adjustment} to the case of non-invertible matrices,
then
\begin{equation}
    \label{eq:rank_1_2}
  \left( \bm{X}_1 - \bm{\mu}_1 \bm{1}^{\intercal} \right)^{\dagger}  = \bm{X}_1^{\dagger} \left( \bm{I} - \frac{\bm{n} \bm{n}^{\intercal}}{\bm{n}^{\intercal} \bm{n}} \right),
\end{equation}
where $\bm{n} = \bm{X}_1^{\dagger^{\intercal}} \bm{1}$ (see Appendix \ref{appendix:rank_one_update}). Plugging \eqref{eq:rank_1_2} into \eqref{eq:rank_one_update}, we find that
\begin{equation*}
  \bar{\bm{A}} = \hat{\bm{A}} - \frac{1}{1 + \norm{\bm{1}^{\intercal} \bm{X}_1^{\dagger} }^2} \bm{X}_2 \bm{X}_1^{\dagger} \bm{X}_1^{\dagger^{\intercal}} \bm{1} \bm{1}^{\intercal} \bm{X}_1^{\dagger}
    = \hat{\bm{A}}
    \left( 
    \bm{I} - \frac{\bm{X}_1^{\dagger^{\intercal}} \bm{1} \bm{1}^{\intercal} \bm{X}_1^{\dagger}}{1 + \norm{\bm{1}^{\intercal} \bm{X}_1^{\dagger} }^2} 
    \right).
\end{equation*} 
Again applying Lemma \ref{lemma:mystery}, 
$\bm{1}^{\intercal} \bm{X}_1^{\dagger} \hat{\bm{A}} = \bm{1}^{\intercal} \bm{X}_1^{\dagger}$,
i.e.\ $\bm{1}^\intercal \bm{X}_1^\dagger$ is a left eigenvector of $\hat{\bm{A}}$ with eigenvalue 1.
By Theorem 2.1 in \cite{ding2007eigenvalues},
we conclude that $\bar{\bm{A}}$ shares all the same eigenvalues $\hat{\bm{A}}$,
except the eigenvalue of $1$ is replaced with $0$.

For the eigenvectors, first note that by 
Lemma \ref{lemma:factor} 
we can express $\bm{X}_1$, $\bm{X}_2$, and $\hat{\bm{A}}$, in terms of the eigenvectors of $\hat{\bm{A}}$,
\begin{equation*}
    \bm{X}_1 = \bm{V} \begin{bmatrix}
    1 & 1 & 1 & \cdots & 1 \\
    1 & \lambda_1 & \lambda_1^2 & \cdots & \lambda_1^{T-1} \\
    \vdots & \vdots & \vdots & & \vdots \\
1 & \lambda_r & \lambda_r^2 & \cdots & \lambda_r^{T-1}
    \end{bmatrix} \mbox{ }
        \bm{X}_2 = \bm{V} \begin{bmatrix}
    1 & 1 & 1 & \cdots & 1 \\
    \lambda_1 & \lambda_1^2 & \lambda_1^3 & \cdots & \lambda_1^T \\
    \vdots & \vdots & \vdots & & \vdots \\
\lambda_r & \lambda_r^2 & \lambda_r^3  &\cdots & \lambda_r^T
    \end{bmatrix} 
\end{equation*}
and
\begin{equation*}
    \hat{\bm{A}} = \bm{V} \begin{bmatrix} 1 & & & \\
& \lambda_1 & & \\
& & \ddots & \\
& & & \lambda_r \\
    \end{bmatrix} \bm{V}^{\dagger}.
\end{equation*}
In this basis, the mean-subtracted data are
\begin{equation*}
    \bar{\bm{X}}_1 =  \frac{\bm{V}}{T} 
    \begin{bmatrix}
    0 & 0 & 0 & \cdots & 0 \\
    1 - \sum_{i = 0}^{T-1} \lambda_1^i & \lambda_1 - \sum_{i = 0}^{T-1} \lambda_1^i & \lambda_1^2 - \sum_{i = 0}^{T-1} \lambda_1^i & \cdots & \lambda_1^{T-1} - \sum_{i = 0}^{T-1} \lambda_1^i \\
    \vdots & \vdots & \vdots & & \vdots \\
1 - \sum_{i = 0}^{T-1} \lambda_r^i & \lambda_r - \sum_{i = 0}^{T-1} \lambda_1^i & \lambda_r^2 - \sum_{i = 0}^{T-1} \lambda_1^i& \cdots & \lambda_r^{T-1} - \sum_{i = 0}^{T-1} \lambda_1^i
    \end{bmatrix}
\end{equation*}
\begin{equation*}
    \bar{\bm{X}}_2 =  \frac{\bm{V}}{T} 
    \begin{bmatrix}
    0 & 0 & 0 & \cdots & 0 \\
    \lambda_1 - \sum_{i = 1}^T \lambda_1^i & \lambda_1^2 - \sum_{i = 1}^T \lambda_1^i & \lambda_1^3 - \sum_{i = 1}^T \lambda_1^i & \cdots & \lambda_1^T - \sum_{i = 1}^T \lambda_1^i \\
    \vdots & \vdots & \vdots & & \vdots \\
\lambda_r - \sum_{i = 1}^T \lambda_r^i & \lambda_r^2 - \sum_{i = 1}^T \lambda_1^i & \lambda_r^3 - \sum_{i = 1}^T \lambda_1^i & \cdots & \lambda_r^T - \sum_{i = 1}^T \lambda_1^i
    \end{bmatrix} .
\end{equation*}
We immediately see that 
\begin{equation*}
    \bm{A}' = \bm{V} \begin{bmatrix} 0 & & & \\ & \lambda_1 & & \\ & & \ddots & \\ & & & \lambda_r \end{bmatrix} \bm{V}^{\dagger}
\end{equation*}
satisfies $\bar{\bm{X}}_2 = \bm{A}' \bar{\bm{X}}_1$. By Theorem \ref{theorem:uniqueness}, the nonzero eigenvalues and corresponding eigenvectors of $\bm{A}'$ must be the same as those of $\bar{\bm{A}}$. 
We conclude that the eigenvectors corresponding to the 
eigenvalues $\lambda_1, \ldots, \lambda_r$ of $\bm{A}$ and $\bar{\bm{A}}$ must be equal up to scaling.


For part 2, suppose $1$ is not an eigenvalue of $\hat{\bm{A}}$. Like before, we can explicitly compute $\bar{\bm{A}}$ as a rank one update to $\hat{\bm{A}}$ (Appendix \ref{appendix:rank_one_update}). 
Since $1$ is not an eigenvalue of $\hat{\bm{A}}$, then by 
Lemma~\ref{lemma:mystery}, 
$\left( \bm{I} - \bm{X}_1^{\dagger} \bm{X}_1 \right)^{\intercal} \bm{1} \neq \bm{0}$. 
We know that, since the data are linearly consistent, 
then the solution $\hat{\bm{A}}$ to DMD without centering satisfies 
$\bm{X}_2 = \hat{\bm{A}} \bm{X}_1$, and thus 
$\bm{X}_2 \left( \bm{I} - \bm{X}_1^{\dagger} \bm{X}_1 \right) = 0$. 
Now,
\begin{equation}
\begin{split}
\hat{\bm{A}} - \bar{\bm{A}} &= \left( \bm{X}_2 \bm{X}_1^{\dagger} - \bar{\bm{X}}_2 \bar{\bm{X}}_1^{\dagger} \right) \\
&= \left( \bm{X}_2 \bm{X}_1^{\dagger} - \bm{X}_2 \left( \bm{I} - \frac{\bm{1} \bm{1}^{\intercal}}{\bm{1}^{\intercal} \bm{1}} \right)  \left( \bm{I} - \frac{ \left( \bm{I} - \bm{X}_1^{\dagger} \bm{X}_1 \right) \bm{1} \bm{1}^{\intercal} }{\bm{1}^{\intercal}\left( \bm{I} - \bm{X}_1^{\dagger} \bm{X}_1 \right) \bm{1} } \right)  \bm{X}_1^{\dagger} \right) \\
&= \left(  \bm{X}_2  \left( \frac{ \left( \bm{I} - \bm{X}_1^{\dagger} \bm{X}_1 \right) \bm{1} \bm{1}^{\intercal} }{\bm{1}^{\intercal}\left( \bm{I} - \bm{X}_1^{\dagger} \bm{X}_1 \right) \bm{1} } \right)  \bm{X}_1^{\dagger} \right) \\
&= \bm{0}.
    \end{split}
\label{eq:rank_one_update_2}
\end{equation}
\end{proof}

\subsection{Linear Systems with Bias}
\label{subsec:linear_systems_with_bias}
We will now illustrate what can go wrong applying DMD
without centering to data generated by a linear system with bias.
Consider the system
\begin{equation*}
    \bm{x}_{j+1} = \begin{bmatrix} 2 & 0 \\ 0 & 3 \end{bmatrix} \bm{x}_j + \begin{bmatrix} 1 \\ 2 \end{bmatrix},
\end{equation*}
with $\bm{x}_1 = [1,1]^{\intercal}$.
This yields data matrices,
\begin{equation*}
    \bm{X}_1 = \begin{bmatrix}
    1 &  3 &  7  \\
    1 &  5 & 17  \\
    \end{bmatrix} \mbox{ and }
        \bm{X}_2 = \begin{bmatrix}
    3 &  7 & 15 \\
    5 & 17 & 53 \\
    \end{bmatrix}.
\end{equation*}
It can easily be shown that these data are linearly inconsistent,
since $\bm{X}_2 \left( \bm{I} - \bm{X}_1^{\dagger} \bm{X}_1 \right) \neq \bm{0}$, 
and so DMD without centering cannot accurately model the data. 

However, for some cases DMD without centering may be able to model data
generated by a linear system with bias, i.e.\ affine dynamics. 
In Theorem~\ref{theorem:dmd_conditions_affine},
we present necessary and sufficient conditions for when this is possible. 
These boil down to 
(1) having ``extra rank'' available to capture the 
bias with an eigenvector of eigenvalue 1 and
(2) a technical condition on the fixed point $\bm{c}$ and the 
eigenvectors of $\bm{A}$ as captured in $\bm{V}$.

\begin{theorem}[DMD without Centering for Affine Systems] 
    Consider data which satisfy the recursive affine equation 
    $\bm{x}_{j+1} = \bm{A} \bm{x}_j + \bm{b}$ for $j = 1, \ldots, T$, such that the affine problem is well-posed. 
    Suppose $\bm{A}$ does not have an eigenvalue equal to $1$, 
    and define the fixed point $\bm{c} = \left( \bm{I} - \bm{A} \right)^{-1} \bm{b}$. 
    Like in \eqref{eq:shift_origin_factor}, we may factor the data $\bm{X}$ as
    \begin{equation*}
    \bm{X} = \begin{bmatrix} \bm{V} & \bm{c} \end{bmatrix} \begin{bmatrix} \bm{\Lambda}^{\intercal}\\ \bm{1}^{\intercal} \end{bmatrix}.  
    \end{equation*}
    Then there exists a diagonalizable $\bm{A}$ such that 
    $\bm{x}_{j+1} = \bm{A} \bm{x}_j$ 
    if and only if $\bm{c}$ is not in the span of the columns of $\bm{V}$.
    \label{theorem:dmd_conditions_affine}
\end{theorem}
\begin{corollary} 
If $\bm{A}$ is full rank and has distinct eigenvalues, 
then DMD without centering will not be able to accurately represent the dynamics. 
\end{corollary}
\begin{proof}
Suppose that $\bm{c}$ is not in the span of the columns of $\bm{V}$.
Lemma \ref{lemma:factor} states that fitting a linear model 
($\bm{x}_{j+1} = \bm{A} \bm{x}_j$) where $\bm{A}$ is diagonalizable is equivalent to the data satisfying
\begin{equation}
  \bm{X} = \tilde{\bm{V}} \tilde{\bm{\Lambda}}^{\intercal},
  \label{eq:tilde_factor}
\end{equation}
where $\tilde{\bm{V}}$ and $\tilde{\bm{\Lambda}}$ have full column rank and $\tilde{\bm{\Lambda}}$ is a rectangular Vandermonde matrix. 
hDefine $\tilde{\bm{V}} = \left[ \bm{V} \mbox{ } \bm{c} \right]$ and 
$\tilde{\bm{\Lambda}} = \left[ \bm{\Lambda} \mbox{ } \bm{1} \right]$.
Then $\bm{X}$ takes the form of \eqref{eq:tilde_factor}. 
By assumption, the columns of $\tilde{\bm{V}}$ are linearly independent,
and, since $1$ is not an eigenvalue of $\bm{A}$, by Lemma \ref{lemma:vandermonde},
the columns of $\tilde{\bm{\Lambda}}$ are linearly independent.
With this factorization $\bm{X}$ satisfies a linear model.
By Theorem \ref{theorem:uniqueness}, reading off from $\tilde{\bm{\Lambda}}$,
the modes of DMD without centering will be the same 
eigenvalues and corresponding eigenvectors of DMD with centering with an additional eigenvalue equal to $1$.

Now suppose that $\bm{c}$ is in the span of the columns of $\bm{V}$. 
To show that a linear system cannot model the data, we will use proof by contradiction.
Suppose that we can express 
\[
[\bm{V} \mbox{ } \bm{c}] [ \bm{\Lambda} \mbox{ } \bm{1}]^{\intercal} 
= \tilde{\bm{V}}\tilde{\bm{\Lambda}}^{\intercal},
\]
where $\tilde{\bm{V}}$ and $\tilde{\bm{\Lambda}}$ have full column rank
(Lemma \ref{lemma:factor}).
Since 
$\mathrm{rank}([\bm{V} \mbox{ } \bm{c}]) 
< 
\mathrm{rank}( [ \bm{\Lambda} \mbox{ } \bm{1}])$, 
this requires that
$\mathrm{rank}(\tilde{\bm{\Lambda}})
< 
\mathrm{rank}( [ \bm{\Lambda} \mbox{ } \bm{1}])$.

For a Vandermonde matrix $\bm{M}$, let
$\sigma(\bm{M})$ denote its ``spectrum,'' i.e.\ the eigenvalues that generate the columns,
so that $\sigma(\bm{\Lambda}) = \{\lambda_1, \ldots, \lambda_r \}$.
We have three cases:
\begin{caseof}
    \case{$\sigma(\tilde{\bm{\Lambda}})$ is a proper subset of $\sigma([ \bm{\Lambda} \mbox{ } \bm{1}])$.}
    In this case, we can partition the columns of $[ \bm{\Lambda} \mbox{ } \bm{1}]$ 
to form two Vandermonde matrices $\tilde{\bm{\Lambda}}$ and $\bm{\Lambda}'$,
where the spectrum of 
$\sigma(\bm{\Lambda}') = 
\sigma([ \bm{\Lambda} \mbox{ } \bm{1} ]) \setminus \sigma(\tilde{\bm{\Lambda}})$. 
Employing the same partition on the columns of $[ \bm{V} \mbox{ } \bm{c} ]$, 
we can form matrices $\bm{V}_1$ and $\bm{V}_2$, such that
\begin{equation}
    \begin{bmatrix} \bm{V}_1 & \bm{V}_2 \end{bmatrix} \begin{bmatrix} \tilde{\bm{\Lambda}}^{\intercal} \\ \bm{\Lambda}'^{\intercal} \end{bmatrix} = \tilde{\bm{V}} \tilde{\bm{\Lambda}}^{\intercal}.
    \label{eq:v1v2}
 \end{equation}
Consider the solution to 
\begin{equation*}
  \bm{B} \begin{bmatrix} \tilde{\bm{\Lambda}}^{\intercal} \\ \bm{\Lambda}'^{\intercal} \end{bmatrix} = \tilde{\bm{\Lambda}}^{\intercal}.
\end{equation*}
By assumption, the columns of 
$[ \tilde{\bm{\Lambda}} \mbox{ } \bm{\Lambda}' ]$ 
are linearly independent, 
and so there is a unique solution
$\bm{B}
=
\tilde{\bm{\Lambda}}^{\intercal} 
\begin{bmatrix} 
    \tilde{\bm{\Lambda}}^{\intercal} \\ 
    \bm{\Lambda}'^{\intercal} 
\end{bmatrix}^{\dagger} 
= 
\left[\bm{I} \mbox{ } \bm{0} \right]$.
Multiplying both sides of \eqref{eq:v1v2} on the right
by 
$\begin{bmatrix} 
    \tilde{\bm{\Lambda}}^{\intercal} \\ 
    \bm{\Lambda}'^{\intercal} 
\end{bmatrix}^{\dagger}$,
we find
\begin{equation*}
 \begin{bmatrix} \bm{V}_1 & \bm{V}_2 \end{bmatrix} 
 = 
 \tilde{\bm{V}} \tilde{\bm{\Lambda}}^{\intercal} 
 \begin{bmatrix} 
    \tilde{\bm{\Lambda}}^{\intercal} \\ 
    \bm{\Lambda}'^{\intercal} 
\end{bmatrix}^{\dagger}  
= \begin{bmatrix} \tilde{\bm{V}} & \bm{0} \end{bmatrix}.
\end{equation*}
However, $\bm{V}_2$ cannot be $\bm{0}$, since it comes from a partition of the nonzero 
columns of $[ \bm{V} \mbox{ } \bm{c} ]$.
    \case{$\sigma(\tilde{\bm{\Lambda}}) = \sigma([ \bm{\Lambda} \mbox{ } \bm{1}])$.}
    The two spectra cannot be equal, since this would imply that 
    $\mathrm{rank}(\tilde{\bm{\Lambda}})
    =
    \mathrm{rank}( [ \bm{\Lambda} \mbox{ } \bm{1}])$.
    \case{There exists some 
    $\lambda \in \sigma( \tilde{\bm{\Lambda}})$ 
    with
    $\lambda \not\in \sigma([ \bm{\Lambda} \mbox{ } \bm{1} ])$.}
    Therefore,
    $\tilde{\bm{\Lambda}}$ 
    contains a column which is not in the column space of 
    $[ \bm{\Lambda} \mbox{ } \bm{1} ]$.
    This means that there is a column in $\tilde{\bm{\Lambda}}$ which is linearly independent
    from the columns of $[ \bm{\Lambda} \mbox{ } \bm{1} ]$.
    Thus, there does not exist a linear combination $\bm{C}$ such that
    \begin{equation*}
      \bm{C} 
        \begin{bmatrix} 
            \bm{\Lambda}^{\intercal} \\ 
            \bm{1}^{\intercal} 
        \end{bmatrix} 
      = \tilde{\bm{V}} \tilde{\bm{\Lambda}}^{\intercal} ,
    \end{equation*}
which is a contradiction.
\end{caseof}
\end{proof}

\subsection{Synthetic Examples}
\label{subsec:synthetic_examples}
\begin{figure}
    \centering
    \begin{overpic}[width=\textwidth]{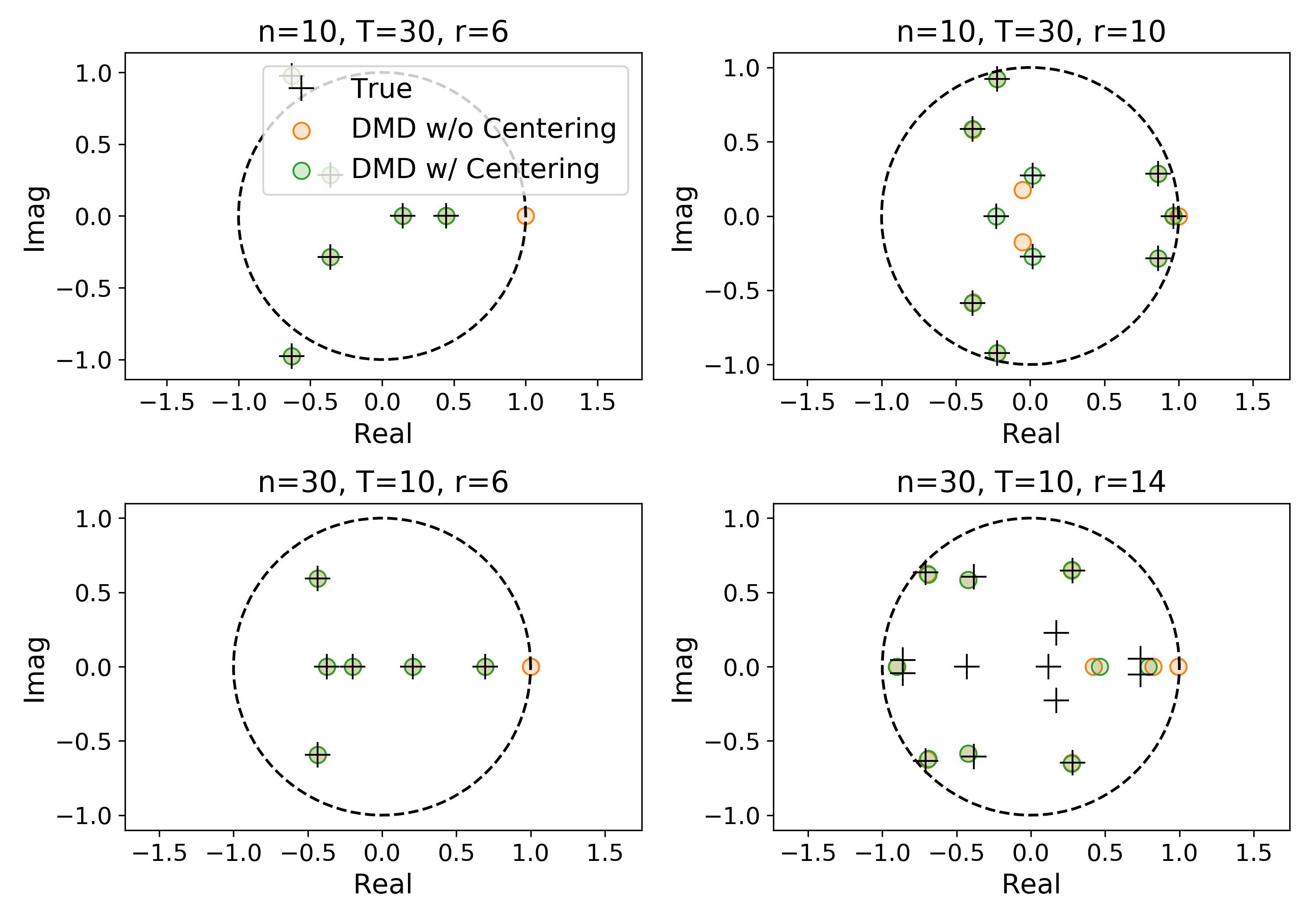}
    \put(3,68){a)}
    \put(3,35){c)}
    \put(50,68){b)}
    \put(50,35){d)}
    \end{overpic}
    \caption{Comparison of the eigenvalues from DMD with centering (green) and DMD without centering (orange) to the true eigenvalues (black) of $\bm{A}$ for four different affine systems $\bm{x}_{j+1} = \bm{A} \bm{x}_j + \bm{b}$. a) $n < T$ and $\bm{A}$ is low rank. DMD with and without centering both yield the eigenvalues of $\bm{A}$, except DMD with centering has an extra eigenvalue of $1$, corresponding to the background mode. b) $n < T$ and $\bm{A}$ is full rank. DMD with centering yields the true eigenvalues of $\bm{A}$ while DMD without centering does not. c) $n > T$ and $\bm{A}$ is low rank. This yields the same result as a). d) Since $T < r$, the DMD problem is not well-posed and neither DMD with centering nor DMD without centering yields the eigenvalues of $\bm{A}$.}
    \label{fig:DMDwithandwithoutcentering}
\end{figure}
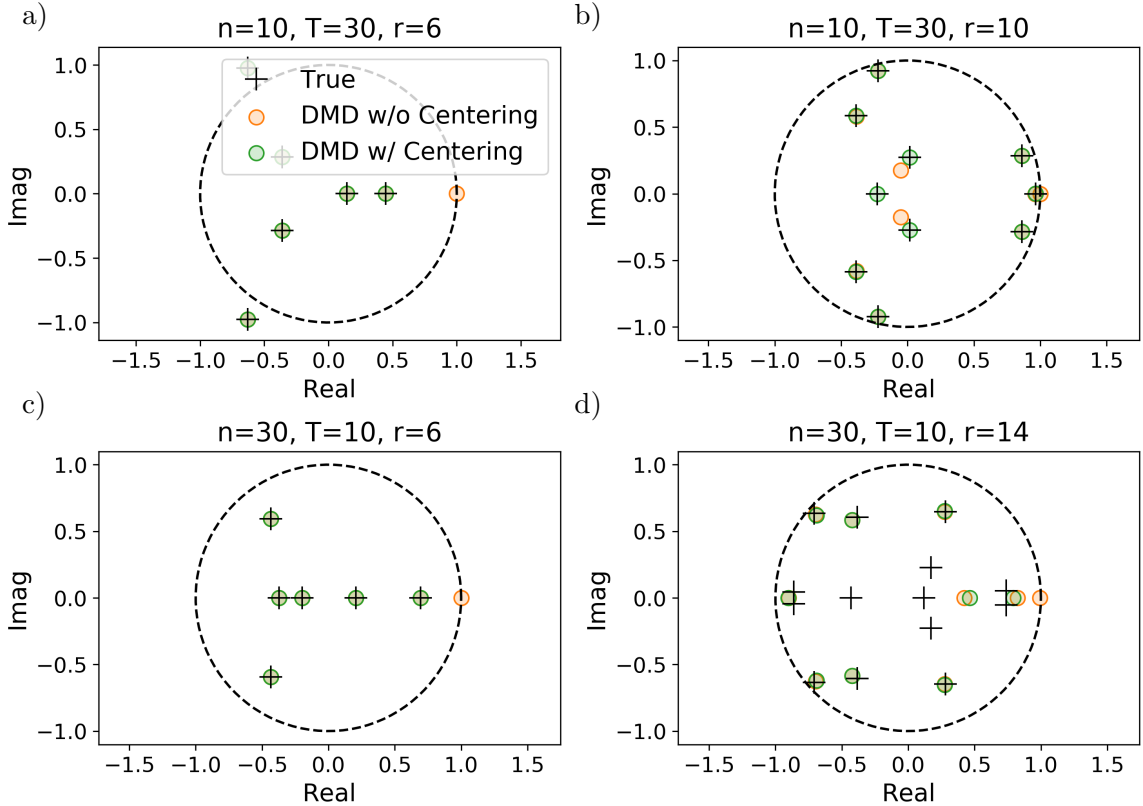
To review our results so far, 
we compare the eigenvalue spectra from DMD with centering and DMD without centering for four sets of measurements of affine systems $\bm{x}_{j+1} = \bm{A} \bm{x}_j + \bm{b}$.
The results are shown in Figure~\ref{fig:DMDwithandwithoutcentering}.
The two top spectra correspond to data with $n < T$, while the bottom two correspond to $n > T$. 

For $n < T$ if $\bm{A}$ is low rank ($r < n$), then DMD without centering has the same spectra as DMD with centering, but with an additional eigenvalue equal to $1$. If $\bm{A}$ is full rank ($r = n$), DMD with centering computes the correct modes. However, DMD without centering cannot accommodate the affine term and does not compute the correct eigenvalues and yields a poor one step reconstruction of $\norm{\bm{X}_2 - \hat{\bm{A}} \bm{X}_1} = 0.019$. 

For $n > T$ with $\bm{A}$ low rank, DMD without centering has the same spectrum as DMD with centering, but with an additional eigenvalue equal to $1$. 
If $r>T$, by Theorem \ref{theorem:uniqueness} the system is under sampled then the DMD problem is not well-posed and the modes of $\bm{A}$ are not unique. Consequently, the DMD modes for both with and without centering do not equal the true modes of $\bm{A}$. That being said, since the data are linearly consistent, all of the models are able to reconstruct the data.

\subsection{The Effects of Noise}
\label{subsec:noise}
In Theorem \ref{theorem:one_rank_update_1} we showed that,
for a linear system,
DMD with centering and DMD without centering will yield the same modes, except that the constant mode without centering is replaced with a zero mode.
However, one of the key assumptions in our proofs is that there exists $\bm{A}$, so that $\bm{x}_{j+1} = \bm{A} \bm{x}_j$. 
For real data with measurement noise, 
this assumption will not hold.
We find empirically that these predictions do hold true, 
with uncertainty on the order of the noise level. 

We simulated data
$
 	\bm{Y} = \bm{X} + \eta \bm{Z},
$
where the elements of $\bm{Z}$ are from a standard normal distribution and perform DMD.
We performed this for 
500 instantiations of $\bm{Z}$ ranging from 
$\eta = 10^{-9}$ to $1$, and with $n=10$, $T = 30$, and $r = 7$.
For both DMD with centering and DMD without centering, 
we compute the sum of the distances from the computed eigenvalues
of $\bar{\bm{A}}$ and $\hat{\bm{A}}$ 
to the nearest true eigenvalue of $\bm{A}$,
then report the median.

Our results are shown in Figure~\ref{fig:noise_analysis}.
The eigenvalue distances, shown at top, for
$\hat{\bm{A}}$ and $\bar{\bm{A}}$ scale linearly with $\eta$. 
These distances are very close for these two methods. 
Note that, when computing the sums for DMD without centering, 
we exclude the eigenvalue closest to one to establish a fairer 
comparison between these two methods. 

For a specific example, in the bottom of
Figure \ref{fig:noise_analysis} 
we plot the eigenvalues computed using these two methods for 
$\eta = 0.005$
over 100 realizations of the noise.
The black crosses show the true eigenvalues. As expected, the deviations of the eigenvalues from the true values are roughly the same. 
Note the presence of the additional eigenvalue equal to 1
for DMD without centering.
\begin{figure}
\centering
\begin{overpic}[width=0.9\textwidth]{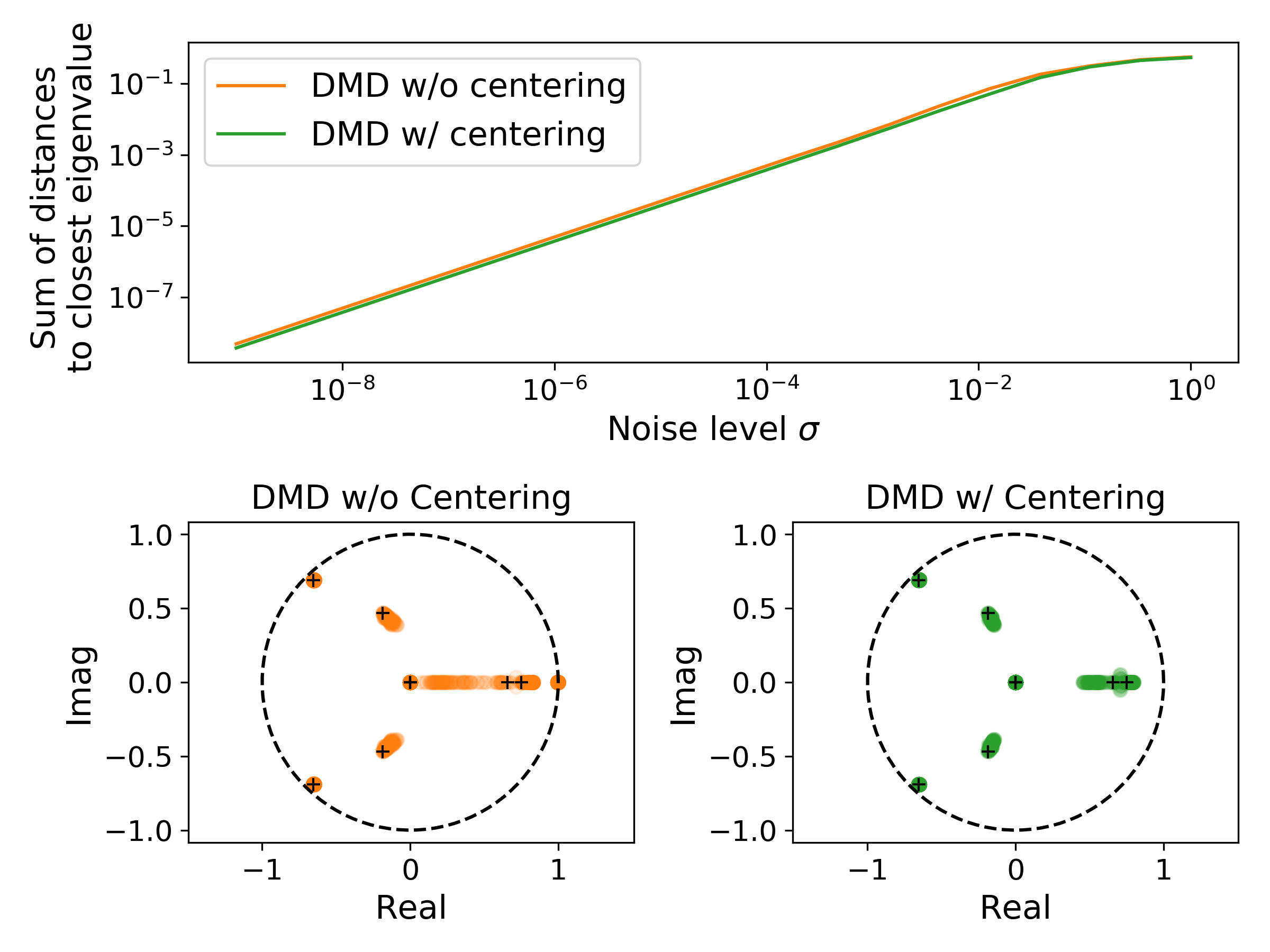}
\put(-5,73){a)}
\put(-5,35){b)}
\end{overpic}
\caption{Comparison of DMD with centering and DMD without centering in the presence of measurement noise. a) For fixed $\bm{A}$, we compute the sum of the distances from the computed eigenvalues of $\bm{A}$ to the nearest true eigenvalue. For both methods, the sum scales linearly with noise level. b) For fixed noise level $0.005$, we plot the eigenvalues of $\bm{A}$ computed using both method for $100$ instantiations of noise. The fluctuations of the eigenvalues from the true values (black crosses) are roughly the same for both methods.}
\label{fig:noise_analysis}
\end{figure}

\section{DMD with Centering is not a Temporal Discrete Fourier Transform}
\label{sec:DFT}

Similar to DMD, the temporal discrete Fourier transform (temporal DFT) can be used to decompose the times series data 
$\bm{x}_1, \ldots \bm{x}_{T+1}$ into linear combinations of modes with exponential time dependence. 
In particular, the temporal DFT is defined as~\cite{daubechies1992ten,harris1978use},
\begin{equation*}
  \hat{\bm{x}}_j := \left\{\frac{1}{T + 1} \sum_{k = 1}^{T+1} \exp{\left( -\frac{2 \pi i (j-1) (k-1)}{T + 1} \right)} \bm{x}_k \right\},
\end{equation*}
with inverse transform 
\begin{equation}
\bm{x}'_k := \left\{ \sum_{k = 1}^{T+1} \exp{\left( \frac{2 \pi i (j-1)(k-1)}{T + 1} \hat{\bm{x}}_k \right) } \right\}.
\label{eq:inverse_fourier_transform}
\end{equation}
In \cite{chen2012variants}, Chen argues that when subtracting the mean $\bm{\mu} = \frac{1}{T+1} \sum_{j = 1}^{T+1} \bm{x}_j$ the eigenvalues $l_j$ of the companion matrix $\bm{C}$ are independent of the data:
\begin{equation*}
  l_j = \exp{\left( \frac{2 \pi i j}{T+1} \right)} \text{ } j = 1, \ldots, T+1.
\end{equation*}
and the eigenvectors $\bm{w}_j$ are given by
\begin{equation*}
  \bm{x}_k = \sum_{j = 2}^{T+1} \exp{\left( \frac{2 \pi i (j-1)(k-1)}{T+1} \right)} \bm{w}_j.
\end{equation*}
Comparing this to \eqref{eq:inverse_fourier_transform} we see that the eigenvectors of the companion matrix correspond to those given by the temporal DFT. Most significantly, this has the unintended consequence of restricting the eigenvalues to be roots of unity. 

It is important to note that this argument is based 
(1) the companion matrix approach and
(2) the  fact that the companion matrix is unique and hence the data matrix 
$\bm{X}_1 - \bm{\mu} \bm{1}^{\intercal}$ has full effective rank. 
Clearly, subtracting the mean $\bm{\mu}_1$ from $\bm{X}_1$ 
will guarantee 
$\bm{X}_1 - \bm{\mu}_1 \bm{1}^{\intercal}$
to be low-rank and therefore have linearly dependent columns. 
Hence, the eigenvalues from the companion matrix approach 
will not equal those from SVD-based DMD, 
and SVD-based DMD will not be equal to the DFT. 
However, if there is a stationary mode, then 
$\bm{X}_1 - \bm{\mu} \bm{1}^{\intercal}$
will also have low effective rank.



\begin{proposition}
Suppose we have sequential time series such that $\bm{x}_{j + 1} = \bm{A} \bm{x}_j$ which are used to define the matrices $\bm{X}$ and $\bm{X}_1$ as in \eqref{eq:x1andx2} and \eqref{eq:Xfactorization}. If
\begin{enumerate}
  \item the DMD problem is well-posed,
  \item $\hat{\bm{A}} = \bm{X}_2 \bm{X}_1^{\dagger}$ has eigenvalue $1$, and
  \item $\bm{X}$ has nonzero mean $\bm{\mu} = \frac{1}{T + 1} \bm{X} \bm{1} \neq 0 $, 
\end{enumerate}
then 
$\mathrm{rank}(\bm{X}_1 - \bm{\mu} \bm{1}^{\intercal})
\leq
\mathrm{rank}(\bm{X}_1) - 1.
$
\end{proposition}

\begin{proof}
To prove that $\bm{X}_1 - \bm{\mu} \bm{1}^{\intercal}$ is rank deficient, we will show that there exists a nonzero vector $\bm{v}$ which is in the nullspace of $\bm{X}_1 - \bm{\mu} \bm{1}^{\intercal}$, but not in the nullspace of $\bm{X}_1$. 

First, we need to show that $\text{range}(\bm{X}_2) \subseteq \text{range}(\bm{X}_1)$. To see this, we only need to consider the case where $\text{range}(\bm{X}_2) \neq \text{range}(\bm{X}_1)$.
Then, since the DMD problem is well-posed, $T \geq r+1$. Thus, $\bm{X}_1$ contains at least $r$ linearly independent vectors which are in the range of $\hat{\bm{A}}$. Thus, $\text{range}(\hat{\bm{A}}) \subseteq \text{range}(\bm{X}_1)$, and since $\bm{X}_2 = \hat{\bm{A}} \bm{X}_1$, then $\text{range}(\bm{X}_2) \subseteq \text{range}(\hat{\bm{A}}) \subseteq \text{range}(\bm{X}_1)$. 

Since, $\text{range}(\bm{X}_2) \subseteq \text{range}(\bm{X}_1)$, then there exists $\bm{c} \in \R^{n-1}$ such that $\bm{X}_1 \bm{c} = \bm{x}_n$. One possible solution to this is $\bm{c} = \bm{X}_1^{\dagger} \bm{x}_{T+1}$. Define $\bm{\alpha} = \frac{\bm{1} + \bm{X}_1^{\dagger} \bm{x}_{T+1}}{T + 1}$. By definition, 
\begin{equation*}
\begin{split}
\bm{0} \neq \bm{\mu} &= \frac{1}{T+1} \bm{X} \bm{1} \\
&=  \frac{1}{T+1} \left( \bm{X}_1 \bm{1}  + \bm{x}_{T+1} \right) \\
&= \frac{1}{T+1}  \bm{X}_1 \left( \bm{1} + \bm{c}  \right) \\
&= \bm{X}_1 \bm{\alpha}.
\end{split}
\end{equation*}
Thus, $\bm{\alpha}$ is not in the nullspace of $\bm{X}_1$ and therefore cannot be $\bm{0}$.
By Lemma \ref{lemma:mystery}, since $\hat{\bm{A}}$ has eigenvalue $1$, then $\bm{1}^{\intercal} \bm{X}_1^{\dagger} \bm{x}_{T+1} = \bm{1}^{\intercal} \bm{c} = 1$. Thus, 
\begin{equation*}
\begin{split}
  \left( \bm{X}_1 - \bm{\mu} \bm{1}^{\intercal} \right) \bm{\alpha} &= \bm{X}_1  \left( \bm{I} - \frac{1}{T+1} \left(\bm{1} + \bm{c} \right) \bm{1}^{\intercal} \right) \frac{(\bm{1} + \bm{c})}{T+1} \\
  &= \bm{X}_1  \frac{(\bm{1} + \bm{c})}{T+1} - \bm{X}_1 \frac{1}{T+1} \left( \bm{1} + \bm{c} \right) \bm{1}^{\intercal}\frac{(\bm{1} + \bm{c})}{T+1} \\
  &= \bm{X}_1  \frac{(\bm{1} + \bm{c})}{T+1} - \bm{X}_1  \frac{(\bm{1} + \bm{c})}{T+1} \\ 
  &= \bm{0}.
\end{split}
\end{equation*}
In conclusion, the dimension of the null space of the centered data $(\bm{X}_1 - \bm{\mu} \bm{1}^{\intercal})$ must be greater than the dimension of the null space of the uncentered data $(\bm{X}_1)$ and the centered data must have a lower rank than the uncentered data.
\end{proof}

\begin{remark} Even in the case where the system has measurement noise, $\bm{X}_1 - \bm{\mu} \bm{1}^{\intercal}$ is effectively rank deficient. Thus, the companion matrix modes are not the DMD modes even if the data is mean subtracted using $\bm{\mu}$.
\end{remark}

To illustrate this point we generate data for an affine system $\bm{x}_{j+1} = \bm{A} \bm{x}_j + \bm{b}$ \eqref{eq:affine} with $n = 10$, $T = 7$, and $r = 5$. 
In Figure \ref{fig:DFT}, we plot the true modes (crosses),
the spectrum computed with DMD with centering and the spectrum computed using the companion matrix on data with the total mean subtracted. We see that DMD with centering extracts out the correct modes. However, the companion matrix approach does not. Since the data $\bm{X}_1 - \bm{\mu} \bm{1}^{\intercal}$ is low rank, the eigenvalues of the companion matrix are not the roots of unity. In addition, note that the companion matrix has seven nonzero eigenvalues even though $\bm{A}$ has only $5$. Next, we add some Gaussian distributed measurement noise with zero mean and standard deviation $0.001$. Like the noiseless case, DMD with centering extracts the correct eigenvalues. Since the data are full rank the companion matrix approach yields a temporal DFT. However, since the data has low effective rank, the eigenvalues of the companion matrix do not equal the eigenvalues of $\bm{A}$. 
\begin{figure}
\centering
\begin{overpic}[scale=0.6,unit=1bp]{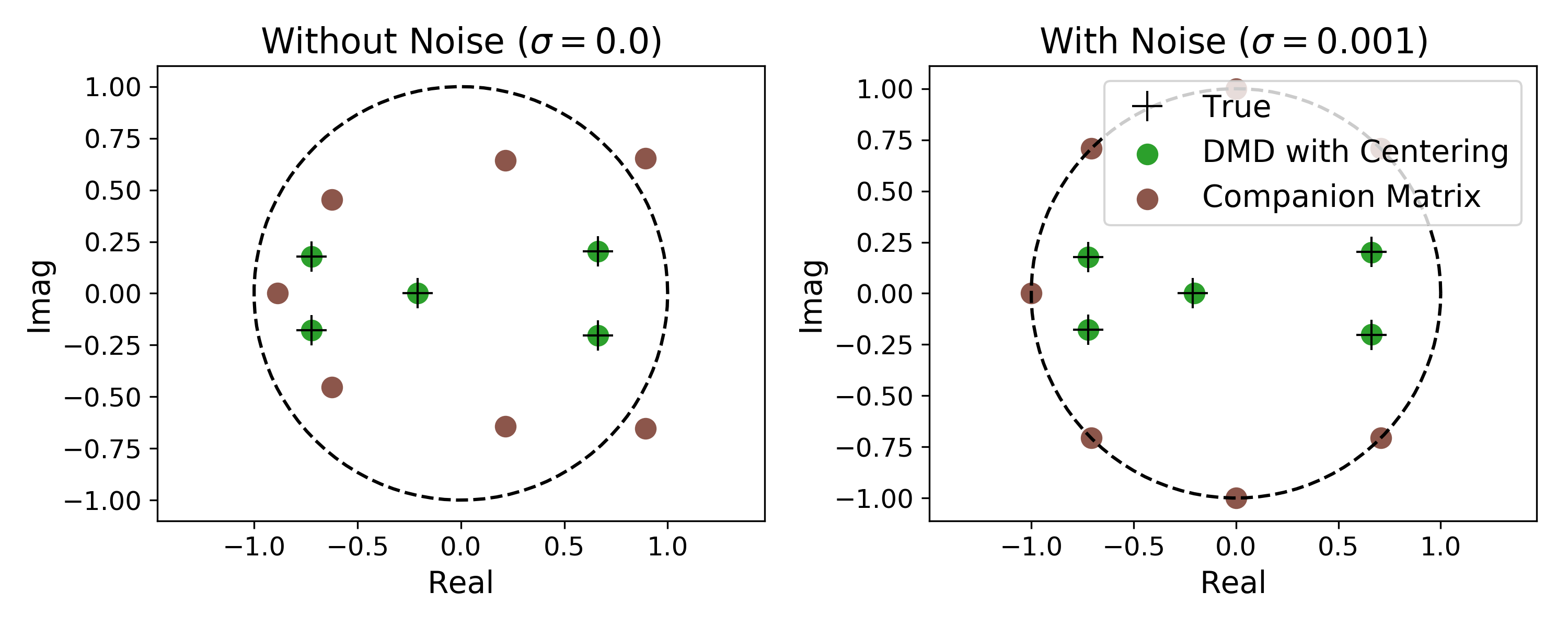}
\put(2,37){a)}
\put(51,37){b)}
\end{overpic}
\caption{Comparison of performance of DMD with centering 
and the companion matrix approach on mean subtracted data. 
a) DMD modes (green) match with the true modes (black crosses).
Since total mean subtracted data 
$\bm{X}_1 - \bm{\mu} \bm{1}^{\intercal}$ is low rank, 
the companion matrix eigenvalues do not equal the true eigenvalues. b) Same system as a) but with added measurement noise. 
DMD with centering yields the correct eigenvalues. 
Since the data is full-rank,
the companion matrix eigenvalues equal the roots of unity. 
However, since the data has low effective rank, 
these modes do not equal the true modes of the system.}
\label{fig:DFT}
\end{figure}





\section{Extracting Arbitrary Frequencies}
\label{sec:arbitraryfrequencies}
By subtracting the means of  $\bm{X}_1$ and $\bm{X}_2$ individually, we have shown we can successfully extract the dynamics about a background mode (corresponding to a DMD mode with eigenvalue $1$). 
We generalize this result to modes with fixed frequencies that correspond to known eigenvalues other than $1$.
As a concrete example, electrical recordings taken in the presence of an alternating current power source are often corrupted with a ``background'' signal at a fixed frequency 
(60 Hz in most countries).
This line noise corresponds to a mode with a precisely known eigenvalue that we want to subtract from the measurements.

To subtract a mode of known frequency, note that in \eqref{eq:affine} the eigenvalue of $1$ comes in through the decision to use $\bm{1}^{\intercal}$. 
By adding this term we enforce that
\begin{equation}
    \bm{1}^{\intercal} = \begin{bmatrix} 1 & 1 \cdots 1\end{bmatrix}, 
    \label{eq:ones}
\end{equation}
appears in the rowspace of the data. We remove this mode by subtracting the mean from the data or equivalently applying the orthogonal projection,
\begin{equation}
    \bm{I} - \frac{\bm{1}\bm{1}^{\intercal}}{\bm{1}^{\intercal} \bm{1}}, 
    \label{eq:orthogonal_projection}
\end{equation}
to $\bm{X}_1$ and $\bm{X}_2$.

If we know that another eigenvalue $\lambda$ exists in the data, then we simply replace \eqref{eq:ones} with
\begin{equation*}
    \bm{\lambda}^{\intercal} = \begin{bmatrix} 1 & \lambda & \lambda^2 & \lambda^3 \cdots \lambda^{T-1} \end{bmatrix}.
\end{equation*}
Thus, \eqref{eq:affine} becomes 
\begin{equation}
    \bm{X}_2 = \bm{A} \bm{X}_1 + \bm{b} \bm{\lambda}^{\intercal}.
    \label{eq:affine_arbitrary}
\end{equation}
Multiplying both sides by $\bm{\lambda}^{T^\dagger} = \frac{\bm{\lambda}^*}{\bm{\lambda}^{\intercal} \bm{\lambda}^*}$, then
\begin{gather*}
    \bm{X}_2 \frac{\bm{\lambda}^*}{\bm{\lambda}^{T^*} \bm{\lambda}} = \bm{b} \frac{\bm{\lambda}^{\intercal} \bm{\lambda}^*}{\bm{\lambda}^{T^*} \bm{\lambda}} + \bm{A} \bm{X}_1 \frac{\bm{\lambda}^*}{\bm{\lambda}^{T^*} \bm{\lambda}} \\
    \bm{b} = \bm{X}_2 \frac{\bm{\lambda}^*}{\bm{\lambda}^{T^*} \bm{\lambda}} - \bm{A} \bm{X}_1 \frac{\bm{\lambda}^*}{\bm{\lambda}^{T^*} \bm{\lambda}}.
\end{gather*}
Plugging this into \eqref{eq:affine_arbitrary}
\begin{equation*}
    \bm{A} \bm{X}_1 \left( \bm{I} - \frac{\bm{\lambda}^* \bm{\lambda}^{\intercal}}{\bm{\lambda}^{\intercal} \bm{\lambda}^*} \right) = \bm{X}_2 \left( \bm{I} - \frac{\bm{\lambda}^* \bm{\lambda}^{\intercal}}{\bm{\lambda}^{\intercal} \bm{\lambda}^*} \right).
\end{equation*}
Thus, solving \eqref{eq:affine_arbitrary} is equivalent to applying the orthogonal projection, $\bm{I} - \frac{\bm{\lambda}^* \bm{\lambda}^{\intercal}}{\bm{\lambda}^{\intercal} \bm{\lambda}^*}$ to the data. 

If there are multiple known distinct eigenvalues $\lambda_1, \ldots, \lambda_k$, then applying the same procedure we construct the matrix 
\begin{equation*}
  \bm{\Lambda}^{\intercal} = \begin{bmatrix} \lambda_1 & \lambda_1^2 & \lambda_1^3 & \cdots & \lambda_1^{\intercal} \\ 
  \lambda_1 & \lambda_1^2 & \lambda_1^3 & \cdots & \lambda_1^T \\
  \vdots & \vdots & \vdots & \ddots & \vdots \\
  \lambda_k & \lambda_k^2 & \lambda_k^3 & \cdots & \lambda_k^T
  \end{bmatrix},
\end{equation*}
and assume that the data satisfies
\begin{equation}
  \bm{X}_2 = \bm{A} \bm{X}_1 + \bm{B} \bm{\Lambda}^{\intercal}.
  \label{eq:big_affine}
\end{equation}
In the case that $k < T$, since $\bm{\Lambda}$ is a Vandermonde matrix it has full column rank (Lemma \ref{lemma:vandermonde}), 
and thus $\bm{\Lambda}^{\dagger} \bm{\Lambda} = \bm{I}$.
Multiplying \eqref{eq:big_affine} by $\bm{\Lambda}^{\intercal^{\dagger}}$, and rearranging terms we get
\begin{equation*}
  \bm{B} = \bm{X}_2 \bm{\Lambda}^{\intercal^{\dagger}} - \bm{A} \bm{X}_1 \bm{\Lambda}^{\intercal^{\dagger}}
\end{equation*}
Plugging this into \eqref{eq:big_affine} yields
\begin{equation*}
  \bm{A} \bm{X}_1 \left( \bm{I} - \bm{\Lambda}^{\intercal^{\dagger}} \bm{\Lambda}^{\intercal} \right) = \bm{X}_2 \left( \bm{I} - \bm{\Lambda}^{\intercal^{\dagger}} \bm{\Lambda}^{\intercal} \right).
\end{equation*}
So, solving \eqref{eq:big_affine} is equivalent to DMD after applying the orthogonal projection $\bm{I} - \bm{\Lambda}^{\intercal^{\dagger}} \bm{\Lambda}^{\intercal}$ to the data.
As an example, we generate data with samples which satisfy,
\begin{equation}
  \bm{x}_{j+1} = \bm{A} \bm{x}_j + \bm{b} \lambda^{j-1} \text{ for } j = 1, \ldots, T+1.
  \label{eq:recursive_affine}
\end{equation}
We choose $n = 10,~T = 9$, $r = 5$ and $\lambda = -i$. The eigenvalues of $\bm{A}$ (black crosses) are shown alongside the eigenvalues computed using DMD without fixing an eigenvalue (orange) and DMD with a fixed eigenvalue (green). As expected, DMD with a fixed eigenvalue extracts out the eigenvalues of $\bm{A}$ while DMD without a fixed eigenvalue includes an additional eigenvalue with value $-i$.

\begin{figure}
    \centering
    \begin{overpic}[unit=1bp,width=\textwidth]{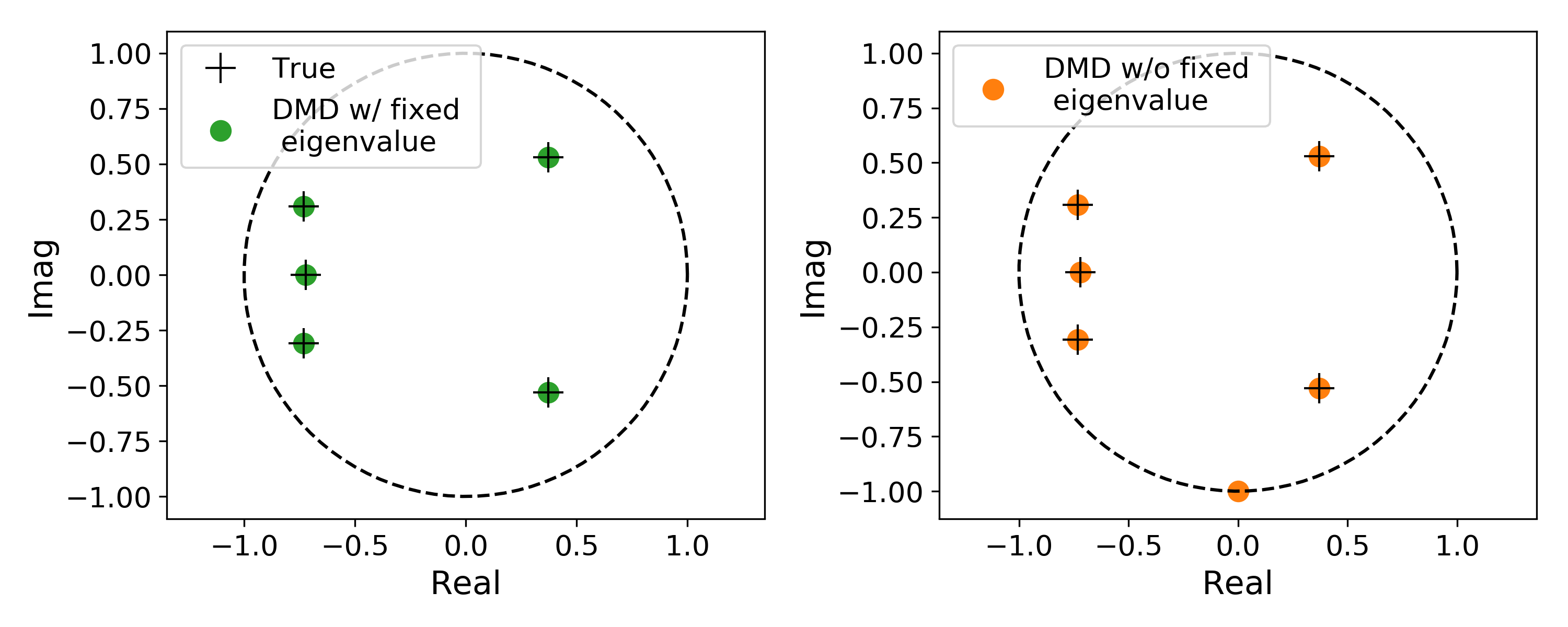}
    \put(2,37){a)}
\put(51,37){b)}
    \end{overpic}
    \caption{Comparison of DMD with fixed 
    frequency subtraction versus ordinary DMD. 
    Data was generated as in \eqref{eq:recursive_affine}, 
    where $\bm{A}$ has six nonzero eigenvalues 
    with one fixed to be $-i$. 
    a) Eigenvalues computed using DMD 
    with fixed frequency subtraction (green) 
    compared to true eigenvalues of $\bm{A}$. 
    b) DMD modes computed without a fixed eigenvalue (orange) 
    compared to true eigenvalues of $\bm{A}$. 
    DMD without fixed frequency subtraction contains the 
    additional eigenvalue equal to $-i$.}
    \label{fig:fixed_eig}
\end{figure}

\section{Examples} \label{sec:examples}

We demonstrate DMD with centering on three nonlinear examples, including one synthetic example and two real-world datasets.
For the Lorenz system, Section~\ref{subsec:lorenz} shows that DMD with centering improves the model of the dynamics, especially in the presence of measurement noise.
Section~\ref{subsec:video} describes the surveillance video example, where the data is effectively low rank; DMD with and without centering extract the same foreground and background modes, as detailed in Section~\ref{sec:comparing_without_centering}.
Last, Section~\ref{subsec:ecog} shows extraction of modes at arbitrary frequencies (Section~\ref{sec:arbitraryfrequencies}) using an example of brain activity recordings contaminated by 60 Hz line noise.

\subsection{Lorenz System} \label{subsec:lorenz}

As an example, we analyze the Lorenz (1963) system~\cite{lorenz1963deterministic} which is defined by the set of differential equations
\begin{align*}
\dot{x}_1 = \sigma (x_2 - x_1) \\
\dot{x}_2 = x_1 \left( \rho - x_3 \right) - x_2 \\
\dot{x}_3 = x_1 x_2 - \beta x_3.
\end{align*}
These equations appear in a variety of systems including, fluid dynamics~\cite{leonov2015homoclinic}, lasers~\cite{weiss1986evidence}, and chemical reactions~\cite{doherty1988chaos}. This system is nonlinear and thus the corresponding data matrix $\bm{X} \in \R^{3 \times T+1}$ has linearly independent rows.

For this analysis, we will focus on applying DMD to a short trajectory that spirals outward from the unstable nonzero fixed point $(\sqrt{\beta (\rho - 1)},\sqrt{\beta (\rho - 1)},\rho - 1)$. We choose to use the common values $\sigma = 10$, $\rho=28$, and $\beta=8/3$. 
We simulate 4800 timepoints using the standard 
Runge-Kutta 4th-order method with fixed timestep $0.001$ and initial condition $\bm{x}_1 = [6.7673,6.1253,25.8706]$.

In Figure \ref{fig:lorenz} we plot the trajectory along 
with the reconstructed trajectories
(forecasts from the initial time)
using DMD with centering 
and DMD without centering.
The corresponding eigenvalues from these two methods are shown on the right. 
DMD with centering and DMD without centering have different eigenvalues in this case. 
However, both methods give similar reconstructions. 
Note that DMD has an eigenvalue very close to 1, 
which indicates that there is a fixed point or nonzero mean in the data.

Next we add Gaussian measurement noise with 
variance $0.03^2$,
which is quite small relative to the variable scales. 
It is well-known that noise shrinks DMD eigenvalues 
towards the origin ~\cite{bagheri2013effects,dawson2016characterizing,hemati2017biasing}. 
For DMD with centering, even though the eigenvalues shrink towards the origin,
the reconstruction is still centered about the fixed point. 
In addition, since two of the DMD with centering eigenvalues remain outside of the unit circle, the reconstructions have the same growing trend as the simulation. 
However, for DMD without centering all of the eigenvalues fall within the unit circle, 
which causes the reconstruction to decay to the origin. So, we conclude that not centering the data can result in drastically different forecasts and estimates of stability.

\begin{figure}
\centering
\begin{overpic}[width=\textwidth,unit=1bp]{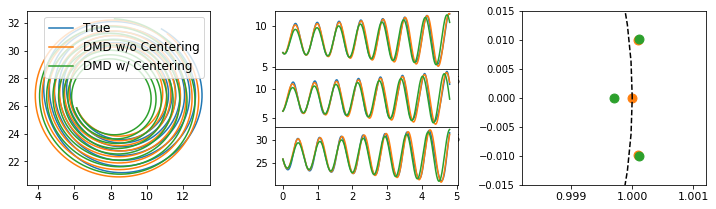}
\put(-1,15){$z$}
\put(34,22){$x$}
\put(34,15){$y$}
\put(34,5){$z$}
\put(64.5,15){Im}
\end{overpic}
\begin{overpic}[scale=0.62,unit=1bp]{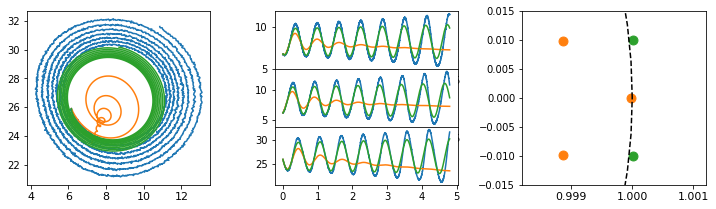}
\put(84,28.5){Re}
\put(50,28.5){$t$}
\put(16,28.5){$y$}
 \put(-1,15){$z$}
\put(16,-0.5){$y$}
\put(34,22){$x$}
\put(34,15){$y$}
\put(34,5){$z$}
\put(50,-0.5){$t$}
\put(64.5,15){Im}
\put(84,-0.5){Re}
\end{overpic}
\caption{Comparison of performance of DMD with and without centering using Lorenz (1963) attractor data. Top left: Reconstruction of $z$ plotted against reconstruction of $y$ for the two methods. 
Top center: Reconstruction of $x$, $y$, and $z$ as a function of $t$ individually using different methods. Both methods produce similar reconstructions. Top right: Eigenvalue spectra for DMD with and without centering. Bottom row: Same as top row except simulation has added Gaussian measurement noise. 
Note that all of the eigenvalues for DMD without centering have magnitude less than one and decay to zero, causing the reconstructed trajectory to decay to zero. However, some of the DMD with centering modes have magnitude greater than one, yielding a better reconstruction. One eigenvalue equal to 0.8866 is not shown for DMD with centering.}
\label{fig:lorenz}
\end{figure}
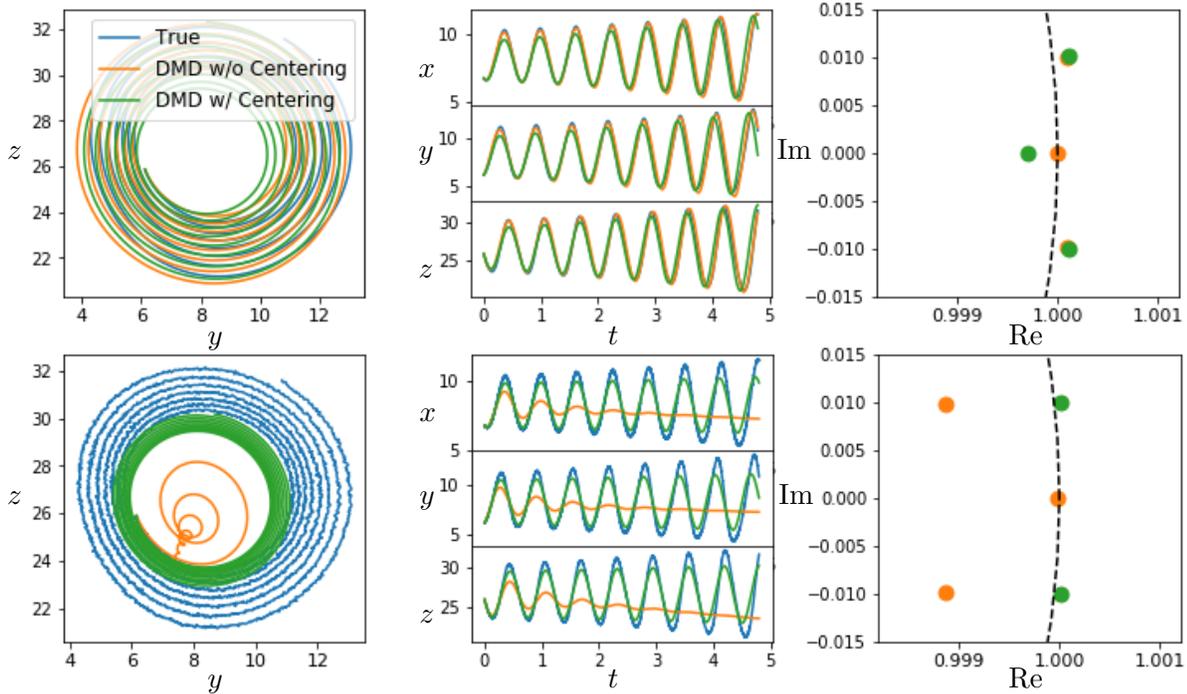

\subsection{Background Subtraction for Video Surveillance} \label{subsec:video}
Next we analyze an application to video surveillance. Here, we focus on one objective, namely foreground/background separation in video. In particular, we would like to split the data into two pieces: a slowly varying, highly correlated background and a foreground containing moving objects of interest.  Several techniques have been developed to address this objective~\cite{tian2005robust,maddalena2008self,he2012incremental,candes2011robust}. In~\cite{grosek2014dynamic}, Grosek and Kutz show that when DMD is applied to videos with static backgrounds, one of the DMD modes typically has an eigenvalue close to $1$. This mode is a good approximation to the background and consequently the difference in the data and this mode yields an estimate for the foreground. For this type of data, we expect the video to be approximately low rank. Hence, the foreground from Exact DMD should be approximately equal to the mean subtracted data. 
In addition, we expect the background to be given by the fixed point $\bm{c}$, where
\begin{equation*}
  \bm{X}_2 - \bm{c} \bm{1}^{\intercal} = \bar{\bm{A}} (\bm{X}_1 - \bm{c} \bm{1}^{\intercal}).
\end{equation*}
Comparing this to the equation for DMD with centering
\begin{equation*}
  \bm{X}_2 - \bm{\mu}_2 \bm{1}^{\intercal} = \bar{\bm{A}} \left( \bm{X}_1 - \bm{\mu}_1 \bm{1}^{\intercal} \right)
\end{equation*}
we see that 
\begin{equation*}
  \bm{c} - \bm{\bar{A}} \bm{c} = \bm{\mu}_2 - \bar{\bm{A}} \bm{\mu}_1,
\end{equation*}
and hence
\begin{equation*}
  \bm{c} = (\bm{I} - \bar{\bm{A}})^{-1} (\bm{\mu}_2 - \bar{\bm{A}} \bm{\mu}_1).
\end{equation*}
Note that $\bm{I} - \bar{\bm{A}}$ is invertible since $\bar{\bm{A}}$ does not have an eigenvalue equal to $1$. 

We apply these methods to surveillance video of highway traffic from the CDNET dataset \cite{wang2014cdnet}. In this case, the foreground is the cars and the background is the grass, road, trees, etc. In Figure \ref{fig:background} we show a sample frame, the stationary mode from DMD without centering, the fixed point $\bm{c}$ from DMD with centering, and the overall mean of the data. 
The stationary mode and fixed point are visually identical but not equal to the overall mean of the data.
Additionally, as predicted by Theorem \ref{theorem:one_rank_update_1}, 
the spectra for DMD with and without centering are nearly identical except for the presence of the additional eigenvalue equal to $1$ for DMD without centering. 
\begin{figure}
\centering
\begin{overpic}[width=\textwidth,unit=1bp]{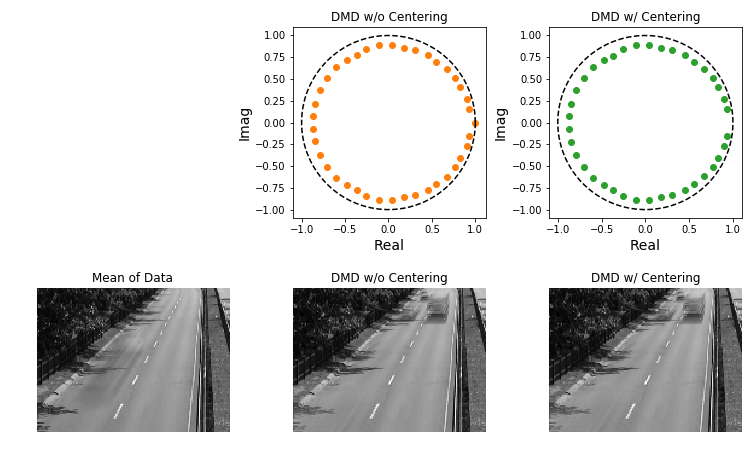}
\put(3,31){\includegraphics[scale=0.11]{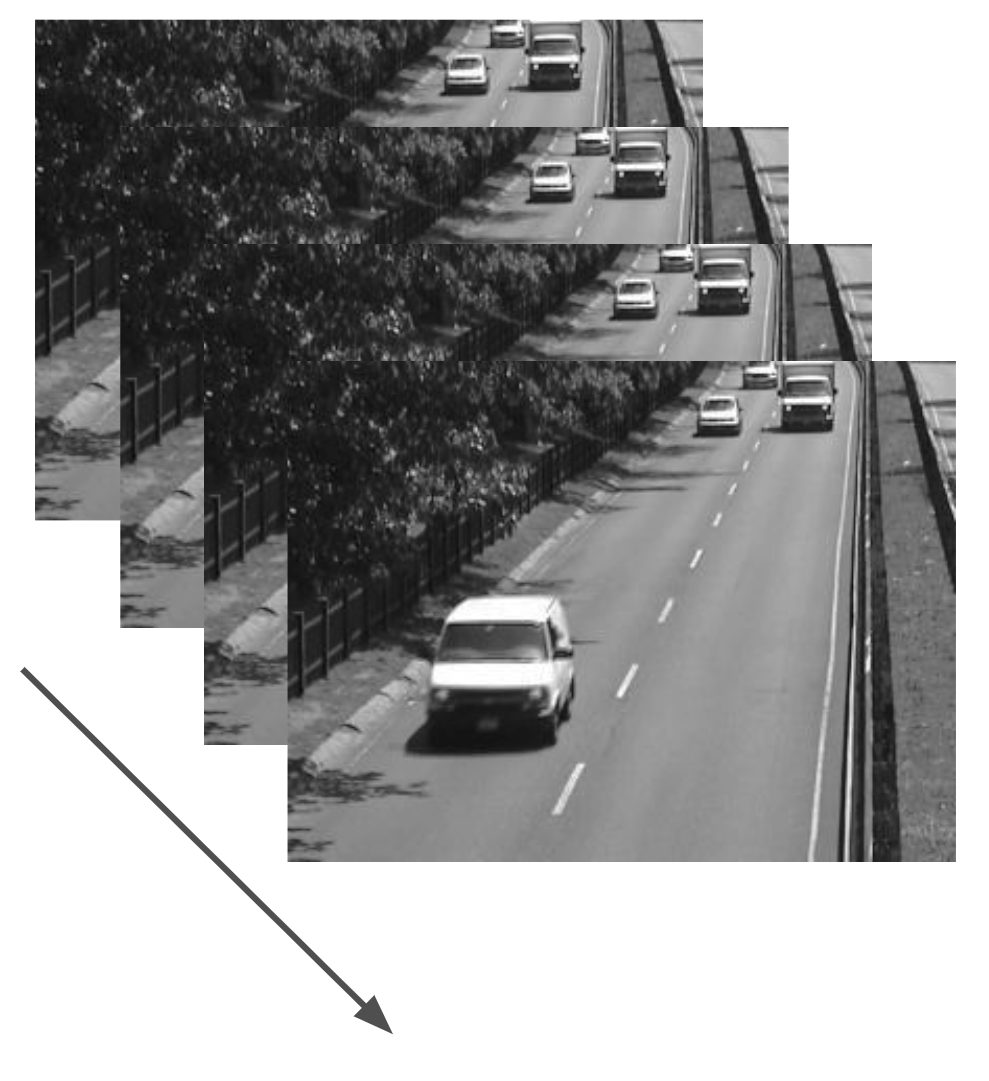}}
\put(6,34){$t$}
\put(0,60){a)}
\put(32,60){c)}
\put(66,60){e)}
\put(0,25){b)}
\put(32,25){d)}
\put(66,25){f)}
\end{overpic}
\caption{a) Sample frames from video of traffic. b) Overall mean $\mu$ of video. c) Eigenvalues of modes computed using DMD and d) static mode corresponding to eigenvalue closest to one. e) Eigenvalues of modes computed using DMD with centering and f) static mode corresponding to fixed point. Note that the spectra for these two methods is nearly identical with the exception of the eigenvalue at $1$ corresponding to the static mode/fixed point.}
\label{fig:background}
\end{figure}
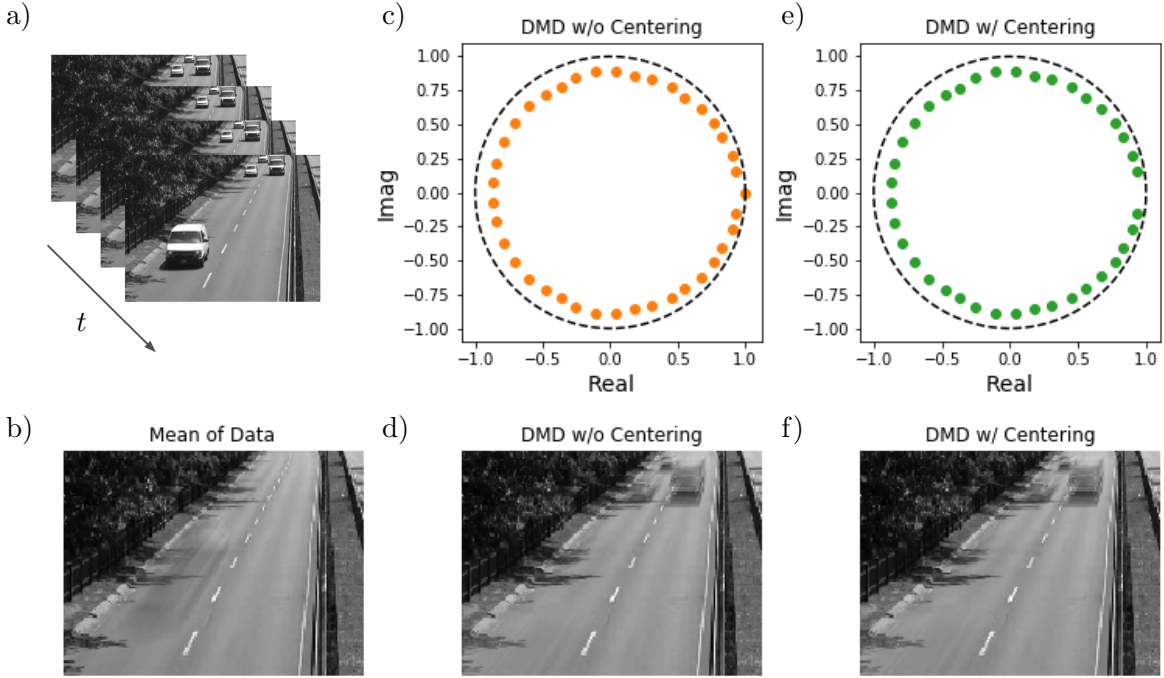

\subsection{Fixed Frequency Subtraction for Brain Activity Recordings} \label{subsec:ecog}
As a final example, we study an application of our methods to brain activity recordings. In particular, we study intracranial electrocorticography (ECoG) measurements from electrodes placed on a human brain surface \cite{wang2018ajile}. 
The data we use contain 
64 channels of measurements 
and are sampled for a duration of 5 seconds 
with a frequency of 1000 Hz.

One common source of signal pollution is 60 Hz power line hum, which results from the AC current in power lines \cite{bai2004adjustable,lee2005design}. 
Since the frequency we would like to remove is known precisely, we may apply the results of Section \ref{sec:arbitraryfrequencies} to denoise the signal.
In the top left of Figure \ref{fig:ecog}, we plot the subset of these channels. 
The corresponding power spectra computed using the temporal DFT 
and DMD are shown in the middle left and bottom left plots,
respectively. 
As expected, there is a distinct peak near 60 Hz in both of these plots. On the right we show the corresponding plots after applying the fixed frequency subraction at 60 Hz. 
In the power spectrum we see that the peak near 60 Hz is suppressed by an order of magnitude. 
In addition, the mode near 60 Hz in the DMD spectrum is completely removed.
Surprisingly, even though the original mode is not at exactly
60 Hz, frequency subtraction is able to remove it.
\begin{figure}
\centering
\begin{overpic}[width=\textwidth,unit=1bp]{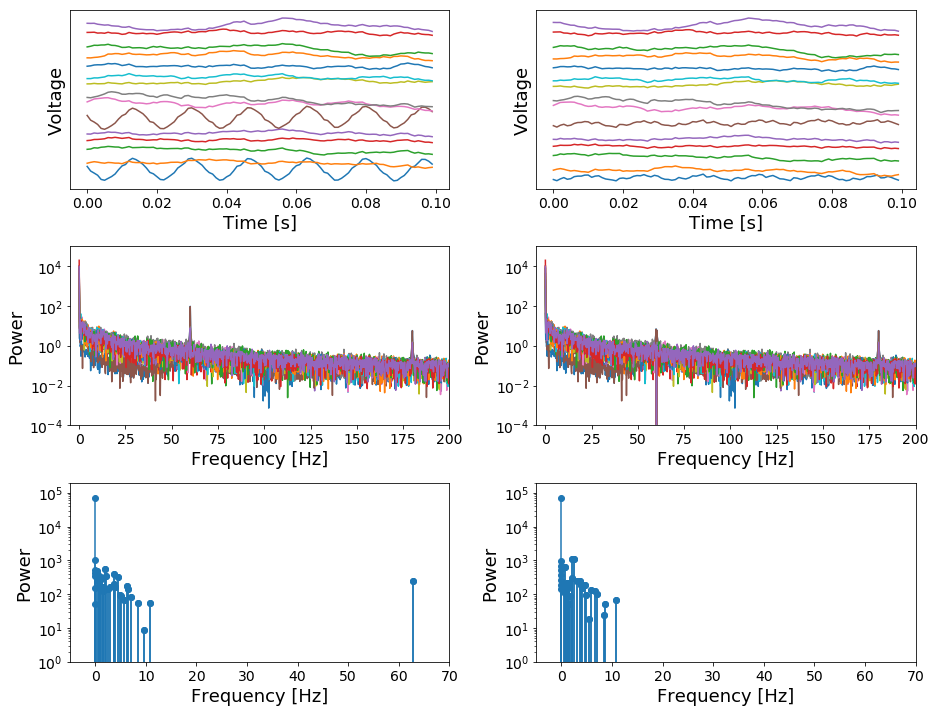}
\end{overpic}
\caption{Application of fixed frequency subtraction to brain activity recordings
. Top left: Raw voltage signals from subset of channels. Center left: Corresponding discrete Fourier transform power spectrum. Bottom left: Power spectrum computed using DMD. Right: Same as left column after fixed frequency of 60 Hz has been subtracted.}
\label{fig:ecog}
\end{figure}

\section{Discussions}

In this paper, we have proposed mean subtraction as a natural and computionally efficient preprocessing step when performing DMD. 
We have shown that DMD on mean subtracted data is equivalent to an additional affine term in the DMD framework, but is not equivalent to a temporal discrete Fourier transform (temporal DFT). 
In addition, we showed that, in a special subset of cases, DMD without centering extracts the same spectra as DMD with centering. 
However, in the case where the data are full rank, DMD with centering can extract the underlying dynamics even when DMD without centering cannot. 
By thinking of centering the data as subtracting a zero-frequency mode, we generalized this result to extracting non-zero, known frequencies in the data. 
Finally, we illustrated DMD with centering on 
three real examples with nonlinear dynamics, 
namely a trajectory of the Lorenz system, 
a surveillance video, and brain recordings.




Many of the Theorems in this work have depended on the assumption of sequential time series sampled at a fixed frequency. 
In particular, the uniqueness of the DMD modes (Theorem~\ref{theorem:uniqueness}) is based on this assumption. 
However, exact DMD has been shown to successfully extract modes from data that is not sequential. 
One potential starting point is the theory of exponential Vandermonde matrices~\cite{robbin2000exponential,yang2001generalization}. 
It remains to be demonstrated that the modes extracted for non-sequential times data by exact DMD, or similar methods such as optimized DMD, are well-posed and unique.
Furthermore, future work remains to more thoroughly explore the 
effects of noise on the DMD estimator and obtain a
fully statistical theory.



Our analysis in this paper has focused on computing DMD 
by what is known as the (SVD-based) 
exact DMD algorithm~\cite{tu2013dynamic}.
There exist many other algorithms for computing the DMD, 
including forward/backward DMD~\cite{dawson2016characterizing}, 
total least squares DMD~\cite{hemati2017biasing}, 
and optimized DMD~\cite{askham2018variable}.
Although we suggest that data centering is generally advantageous,
the consequences of centering remains to be explicitly 
characterized when using these other algorithms.


\section*{Acknowledgments}
We are grateful for discussions with S. H. Singh and S. L. Brunton; and to S. M. Peterson for help with the brain recording dataset. 
This work was funded by the Air Force Office of Scientific Research (FA9550-17-1-0329 to JNK); the Air Force Research Lab (FA8651-16-1-0003 to BWB); the National Science Foundation (award 1514556 to BWB); the Alfred P. Sloan Foundation and the Washington Research Foundation to BWB.
KDH was supported by a Washington Research Foundation Postdoctoral Fellowship.

\bibliography{centered.bib} 
\bibliographystyle{siamplain}

\appendix
\section{Rank one Update}
\label{appendix:rank_one_update}
Here we will derive \eqref{eq:rank_one_update} and \eqref{eq:rank_one_update_2}. Namely we will show that if $\bm{X}_2 = \bm{A} \bm{X}_1$ then 
\begin{equation}
 \bar{\bm{X}}_1^{\dagger} = 
  \begin{cases}
  \bm{X}_1^{\dagger} 
  \left( \bm{I} - \frac{\bm{n} \bm{n}^{\intercal} }{\bm{n}^{\intercal} \bm{n} } \right) & \text{if }  (\bm{I} - \bm{X}_1^{\dagger} \bm{X}_1)^{\intercal} \bm{1} = \bm{0} \\
  \left( \bm{I} - \frac{\left( \bm{I} - \bm{X}_1^{\dagger} \bm{X}_1 \right) \bm{1} \bm{1}^{\intercal} }{\bm{1}^{\intercal} \left( \bm{I} - \bm{X}_1^{\dagger} \bm{X}_1 \right) \bm{1} } \right) \bm{X}_1^{\dagger} & \text{ otherwise }
  \end{cases}
  \label{eq:x1_pseudoinverse}
\end{equation}
where $\bm{n} = \bm{X}_1^{\dagger^{\intercal}} \bm{1}$.
To derive this we use the rank-one update formula (3.2.7) from \cite{petersen2008matrix} to compute $\bar{\bm{X}}_1^{\dagger} = \left( \bm{X}_1 - \bm{\mu}_1 \bm{1}^{\intercal} \right)^{\dagger}$. 

First, let's assume $\left( \bm{I} - \bm{X}_1^{\dagger} \bm{X}_1 \right) \bm{1} = \bm{0}$. Letting $A = \bm{X}_1$, $c = -\bm{\mu}_1$, and $d = \bm{1}$, then
\begin{gather*}
    \beta = 1 - \bm{1}^{\intercal} \bm{X}_1^{\dagger} \bm{\mu}_1 = 1 - \frac{\bm{1}^{\intercal} \bm{X}_1^{\dagger} \bm{X}_1 \bm{1}}{\bm{1}^{\intercal} \bm{1}} = 0 \\
    \bm{w} = - (\bm{I} - \bm{X}_1 \bm{X}_1^{\dagger}) \frac{\bm{X}_1 \bm{1}}{\bm{1}^{\intercal} \bm{1}} = \bm{0} \\
    \bm{m} = (\bm{I} - \bm{X}_1^{\dagger} \bm{X}_1)^{\intercal} \bm{1} = \left( \bm{1}^{\intercal} - \bm{1}^{\intercal} \bm{X}_1^{\dagger} \bm{X}_1 \right)^{\intercal} = \bm{0} \\
    \bm{v} = -\bm{X}_1^{\dagger} \bm{\mu}_1 = -\bm{X}_1^{\dagger} \frac{\bm{X}_1 \bm{1}}{\bm{1}^{\intercal} \bm{1}} = -\frac{\bm{1}}{\bm{1}^{\intercal} \bm{1}} \\
    \bm{n} = \bm{X}_1^{\dagger^{\intercal}} \bm{1} 
\end{gather*}
Note that $\norm{\bm{v}}^2 = \frac{1}{\bm{1}^{\intercal} \bm{1}}$. Since $\beta = \norm{\bm{m}} = \norm{\bm{w}} = 0$, we are in Case $6$ and the pseudoinverse is given by 
\begin{equation*}
     \bar{\bm{X}}_1^{\dagger} = \bm{X}_1^{\dagger} -\frac{1}{\norm{\bm{v}}^2} \bm{v} \bm{v}^{\intercal} \bm{X}_1^{\dagger} - \frac{1}{\norm{\bm{n}}^2} \bm{X}_1^{\dagger} \bm{n} \bm{n}^{\intercal} + \frac{\bm{v}^{\intercal} \bm{X}_1^{\dagger} \bm{n}}{\norm{\bm{v}}^2 \bm{\norm{n}}^2} \bm{v} \bm{n}^{\intercal} 
\end{equation*}

Noting that in first term $\frac{\bm{v}^{\intercal} \bm{X}_1^{\dagger}}{\norm{\bm{v}}^2} = \bm{n}^{\intercal}$ and in the third term $\bm{v}^{\intercal} \bm{X}_1^{\dagger} \bm{n} = \frac{\bm{1}^{\intercal} \bm{X}_1^{\dagger} \bm{X}_1^{\dagger^{\intercal}} \bm{1}}{\bm{1}^{\intercal} \bm{1}} = \norm{\bm{n}}^2  \norm{\bm{v}}^2$,
 then the first and third terms equal $-\bm{v} \bm{n}^{\intercal}$ and $\bm{v} \bm{n}^{\intercal}$. These cancel, yielding the first case of \eqref{eq:x1_pseudoinverse}.
 Thus, by Theorem 2.1 in \cite{ding2007eigenvalues}, $\hat{\bm{A}}$ and $\A$ share all the same eigenvalues and eigenvectors except the eigenvalue of $\hat{\bm{A}}$ equal to $1$ which becomes
 \begin{equation*}
     1 - \frac{\bm{n}^{\intercal} \bm{X}_2 \bm{X}_1^{\dagger} \bm{n}}{\norm{\bm{n}}^2} = 1 - \frac{\bm{n}^{\intercal} \hat{\bm{A}} \bm{n}}{\norm{\bm{n}}^2} = 1 - \frac{\bm{n}^{\intercal} \bm{n}}{\norm{\bm{n}}^2} = 0.
\end{equation*}

Now, let's assume $\bm{1}^{\intercal} \left( \bm{I} - \bm{X}_1^{\dagger} \bm{X}_1 \right) \neq \bm{0}$. This corresponds to Case 3 in \cite{petersen2008matrix}.
\begin{gather*}
  \bm{w} = -\left( \bm{I} - \bm{X}_1 \bm{X}_1^{\dagger} \right) \frac{\bm{X}_1 \bm{1}}{\bm{1}^{\intercal} \bm{1}} = \bm{0} \\
  \bm{m} = \left( \bm{I} - \bm{X}_1^{\dagger} \bm{X}_1 \right)^{\intercal} \bm{1} \neq \bm{0} \\
  \beta = 1 - \frac{\bm{1}^{\intercal} \bm{X}_1^{\dagger} \bm{X}_1 \bm{1}}{\bm{1}^{\intercal} \bm{1}} = \frac{\norm{\bm{m}}^2}{\bm{1}^{\intercal} \bm{1}} \neq 0 \\
  \bm{v} = -\frac{\bm{X}_1^{\dagger} \bm{X}_1 \bm{1}}{\bm{1}^{\intercal} \bm{1}} \\
  \bm{n} = \bm{X}_1^{\dagger^{\intercal}} \bm{1}
\end{gather*}
\begin{equation*}
  \begin{split}
  \bar{\bm{X}}_1^{\dagger} &= \bm{X}_1^{\dagger} + \frac{1}{\beta} \bm{m} \bm{v}^{\intercal} \bm{X}_1^{\dagger} - \frac{\beta}{\norm{\bm{v}}^2 \norm{\bm{m}}^2 + \abs{\beta}^2} \left( \frac{\norm{\bm{v}}^2}{\beta} \bm{m} + \bm{v} \right) \left( \frac{\norm{\bm{m}}^2}{\beta}\left( \bm{X}_1^{\dagger} \right)^{\intercal} \bm{v} + \bm{n} \right)^{\intercal}
  \end{split}
\end{equation*}
Now 
\begin{equation*}
 \frac{\norm{\bm{m}}^2}{\beta} \left( \bm{X}_1^{\dagger} \right)^{\intercal} \bm{v} + \bm{n}  = -\bm{1}^{\intercal} \bm{1} \bm{X}_1^{\dagger^{\intercal}} \bm{X}_1^{\dagger} \frac{\bm{X}_1 \bm{1}}{\bm{1}^{\intercal} \bm{1}} + \bm{X}_1^{\dagger^{\intercal}} \bm{1} = \bm{0} \\
\end{equation*}
since $\bm{X}_1^{\dagger} \bm{X}_1$ is symmetric. Hence,
\begin{equation*}
\begin{split}
  \bar{\bm{X}}_1^{\dagger} &= \bm{X}_1^{\dagger} +  \frac{1}{\beta} \bm{m} \bm{v}^{\intercal} \bm{X}_1^{\dagger} \\
  &= \left( \bm{I} - \frac{\left( \bm{I} - \bm{X}_1^{\dagger} \bm{X}_1 \right) \bm{1} \bm{1}^{\intercal} }{\bm{1}^{\intercal} \left( \bm{I} - \bm{X}_1^{\dagger} \bm{X}_1 \right) \bm{1} } \right) \bm{X}_1^{\dagger}.
\end{split}
\end{equation*}

\end{document}